\documentclass[11pt,a4paper]{amsart}
\usepackage[utf8x]{inputenc}
\usepackage{ucs}
\usepackage{amsmath}
\usepackage{amsthm}
\usepackage{amsfonts}
\usepackage{amssymb}
\usepackage{geometry}
\usepackage{hyperref}
\usepackage{dsfont}
\usepackage{color}
\usepackage{graphicx}
\usepackage[foot]{amsaddr}

\newtheorem{prop}{Proposition}
\newtheorem{rem}{Remark}
\newtheorem{exa}{Example}
\newtheorem{coro}{Corollary}
\newtheorem{thm}{Theorem}

\newtheorem{lemma}{Lemma}
\newcommand{\R}{\mathbb{R}}
\newcommand{\E}{\mathbb{E}}
\newcommand{\s}{\scalebox{1.3}{${}^\sharp$ \hspace{-0.15 cm}}}

\title[On the absolute continuity of random nodal volumes]{On the absolute continuity \\of random nodal volumes}
\begin{document}

\author{J\"urgen Angst}
\author{Guillaume Poly}
\address{Univ Rennes, CNRS, IRMAR - UMR 6625, F-35000 Rennes, France}
\email{jurgen.angst@univ-rennes1.fr, guillaume.poly@univ-rennes1.fr}
  

\begin{abstract}
We study the absolute continuity with respect to the Lebesgue measure of the distribution of the nodal volume associated with a smooth, non-degenerated and stationary Gaussian field $(f(x), {x \in \mathbb R^d})$. Under mild conditions, we prove that in dimension $d\geq 3$, the distribution of the nodal volume has an absolutely continuous component plus a possible singular part. This singular part is actually unavoidable baring in mind that some Gaussian processes have a positive probability to keep a constant sign on some compact domain. Our strategy mainly consists in proving closed Kac--Rice type formulas allowing one to express the volume of the set $\{f =0\}$ as integrals of explicit functionals of $(f,\nabla f,\text{Hess}(f))$ and next to deduce that the random nodal volume belongs to the domain of a suitable Malliavin gradient. The celebrated Bouleau--Hirsch criterion then gives conditions ensuring the absolute continuity.
\end{abstract}

\maketitle

\par
\vspace{1cm}
\tableofcontents

\newpage
\section{Introduction}

Nodal sets, i.e. vanishing loci of functions, are central objects in mathematics. They are for example at the very definition of algebraic varieties and thus the main object of algebraic geometry, but they also appear naturally in analysis, differential geometry and mathematical physics
Understanding the main features of a purely deterministic nodal set is generally out of reach, as illustrated by several celebrated open problems, such as Hilbert’s sixteenth problem or else Yau’s conjecture. In order to capture the typical behavior of an object, one is thus tempted to randomize, which reduces here to consider nodal sets associated with random functions. Computing expected values, variances or else fluctuations around the mean of the considered nodal functionals, and in particular understanding their asymptotic behavior as the amount of noise goes to infinity, is then a true wealth of information about the possible deterministic behaviors. Besides, randomization of nodal sets is also strongly motivated by deep physical insights, such as the celebrated Berry’s conjecture as explained in \cite{Ber1977}.
\par
\bigskip
Given a smooth function $f$ from $\R^d$ to $\mathbb R$, the associated nodal set $\{ f=0\}$ is generically a smooth submanifold of dimension $d-1$. One of the simplest nodal observable is then the nodal volume $\mathcal H^{d-1}(\{f=0\} \cap K)$, that is, the $d-1$ Hausdorff measure of the zero set in some compact domain $K \subset \mathbb R^d$. Undoubtedly, the most important tool in studying such functionals is the so-called Kac--Rice formula. Let us recall that, if the function $f$ is non-degenerated in the sense that 
\begin{equation}\label{eq.non.deg}
\min_{x \in K} \eta_f(x)>0, \quad \mathrm{where} \quad \eta_f(x):=\sqrt{f^2(x)+\|\nabla f(x)\|_2^2},
\end{equation}
i.e. the zeros of $f$ are not critical points, the Kac--Rice formula  reads
\begin{equation}\label{Kac--Rice-Intro}
\mathcal{H}^{d-1}\left( \{f=0\}\cap K\right)=\lim_{\epsilon\to 0}\frac{1}{2\epsilon}\int_{K} \textbf{1}_{\{|f(x)|<\epsilon\}}||\nabla f(x)||_2 \mathcal H_d(dx).
\end{equation}

When $f$ is a smooth random Gaussian field, one can then naturally take the expectation in Equation \eqref{Kac--Rice-Intro} and get rather easily the exact value of $\mathbb{E}\left[\mathcal{H}^{d-1}\left(\{f=0\}\cap K\right)\right]$. It is also possible to compute higher moments with similar formulas, which in that case can become more intricate. We refer the reader to the book \cite{azais2009level} for a rather exhaustive exposition of Kac--Rice type methods and applications.
\par
\bigskip
In the framework of random Gaussian fields, another salient tool is the so-called Wiener-chaotic expansion of the random variable $\mathcal{H}^{d-1}\left(\{f=0\}\cap K\right)$. This tool has shown recently its great efficiency in order to establish limit Theorems, both central and non central, regarding the random nodal volume. The literature on this topic is blowing and the reader can consult the following non-exhaustive list of related works \cite{NV1992,KL1997,AL2013,ADL2015,MPRW2016,AADL2017,cammarota2018quantitative,nourdin2017nodal} as well as the nice survey \cite{rossi2018random} for an overview. For each of the aforementioned models, it is indeed possible to compute explicitly the expansion into Wiener chaoses, which can be seen formally as the infinite dimensional analogue of the Hermite polynomials. The exact computation of the chaotic expansions enables one to use the so-called \textit{Nualart--Peccati criterion} \cite{nualart2005central} as well as \textit{Peccati--Tudor Theorem} \cite{peccati2005gaussian} which both provide efficient criteria ensuring central convergence.
\par
\newpage
In this article we provide a new tool enabling to study the probabilistic properties of the nodal volume associated with a smooth stationary Gaussian field. Namely, we first exhibit some deterministic closed Kac--Rice type formulas for the nodal volume of non-degenerated functions. Then, we deduce that, under mild conditions, the random nodal volume associated with a smooth stationary Gaussian field belongs to the domain of a Malliavin derivative operator. The Malliavin derivative has shown to be an efficient tool for proving the existence of densities for the law of random variables which belong to the domain of the Malliavin gradient. This theory is mainly due to P. Malliavin which used it to give an alternative proof of the hypoellipticiy criterion of H\"{o}rmander. Aside from this emblematic consequence of Malliavin theory, the reader is referred to \cite{nualart2006malliavin}  or \cite{nourdin2010stein} for many other applications of Malliavin calculus. In this article, we will rely on the so-called Bouleau-Hirsch criterion of existence of densities, see e.g. \cite[page 42]{bouleau2003error}. This criterion is more general compared to the Malliavin criterion since it requires less non-degeneracy, however it does not give any information on the regularity of the density.

\par
\medskip
Let us describe in more details our strategy and the organization of the paper. The first step, which is the object of the whole Section 2 below, consists in rewriting the deterministic formula \eqref{Kac--Rice-Intro} as an explicit closed formula taking the form of an integral of a simple functional of both the function $f$ and its derivatives. To the best of our knowledge, this above mentionned closed formula and its variants seem to be new and are of independent interest. Sections 2.1 to 2.3 are devoted the one-dimensional framework, whereas Sections 2.4 and 2.5 deal with the higher dimensional setting. 
\if{
\textcolor{blue}{For example, in Corollary \ref{crazy-formula} of Section 2.2 below, we will establish that if $f$ is a $C^2$ periodic function from $\mathbb T:=\mathbb R/\mathbb Z$ to $\mathbb R$, which is non-degenerated in the sense \eqref{eq.non.deg}, its total number of zeros is given by the exact formula
\[
 \mathcal H^{0}\left(\{f=0\}\right) =\displaystyle{\frac{1}{\pi} \int_0^{2\pi} \frac{f'(x)^2-f(x) f''(x)}{f^2(x)+f'(x)^2} }dx.
 \]
In the same manner, if $f : \mathbb R \to \mathbb R$ is non-degenerated, and if $a<b$ are such that we have $f(a)f(b)\neq 0$, we will see in Section 2.3 that
\[
\mathcal H^0\left(\{f=0\} \cap [a,b]  \right) =\frac{1}{\pi} \left[ \arctan \frac{f'(b)}{f(b)}-\arctan \frac{f'(a)}{f(a)} +   \int_a^b \frac{f'^2(x)-f(x)f''(x)}{f^2(x)+f'^2(x)}dx\right].
\]
The object of the next Sections 2.4 and 2.5 is then to establish several similar formulas in a higher dimensional setting. }}
\fi
To give to the reader a foretaste of the closed formulas we have in mind, if $f$ is a smooth, non-degenerated, periodic function from $\mathbb T^d=\mathbb R^d/\mathbb Z^d$ to $\mathbb R$, we will prove for example in Proposition \ref{prop.IPP2} below that
\begin{equation*}
\begin{array}{ll}
\mathcal H^{d-1}\left(\{f=0\}\right) & =\displaystyle{ - \frac{1}{2} \int_{\mathbb T^d}  \left( f(x) \Delta f(x)-  ||\nabla f(x)||^2 \right) \frac{|f(x)|}{\eta_f(x)^3} dx}\\
\\
&  \displaystyle{+\int_{\mathbb{T}^d} |f(x)| \left(\|\text{Hess}_x (f)\|^2-\text{Tr}\left(\text{Hess}_x(f)\right)^2\right)\frac{dx}{\eta_f(x)^3}}\\
\\
& \displaystyle{+\frac{3}{2} \int_{\mathbb{T}^d}\frac{|f(x)|}{\eta_f^5(x)} \left(\Delta f(x) \langle \nabla f(x), \nabla \eta_f^2(x)\rangle -\nabla f(x)^* \mathrm{Hess}_x f \nabla \eta_f^2(x)\right) dx.}
\end{array}
\end{equation*}
Compared to the classical Kac--Rice formula \eqref{Kac--Rice-Intro}, the suppression of the limit in this last formula, and the fact that each of the three integrands is a Lipschitz functional of the vector $(f, \nabla f, \mathrm{Hess}(f))$, will allow us to deploy the rich properties of stability of the domain of the Malliavin gradients. One can also use this formula to provide moment estimates for the nodal volume. Relying on the proof of lemma \ref{Integrand is in the domain}, we may infer that $\mathcal H^{d-1}\left(\{f=0\}\right)$ has a moment of any order less than $\frac{d+1}{2}$ provided that the Gaussian process is of class $\mathcal{C}^2(\mathbb{T}^d,\mathbb{R})$ and non-degenerated. Interestingly, this observation complements the Theorem 5 from the recent preprint \cite{AAGL18} which basically asserts that the nodal volume of a suitable $\mathcal{C}^\infty$ Gaussian field has moments of every order. Our formula establishes that, under the assumption of two derivatives, the integrability of the nodal volume increases as the dimension increases, see the remark \ref{moment}.
\par
\medskip
In Section 3, we give a first immediate application of these closed Kac--Rice type formulas associated with the celebrated Birkhoff ergodic theorem, namely we establish a strong law of large number for the number of zeros of a smooth, one-dimensional, stationary Gaussian process on a large interval. To the best of our knowlegde, this strong law of large numbers does not appear in the litterature. However, variance estimates, see e.g. \cite{wigman2011} naturally entail the convergence in probability of the properly renormalized number of zeros.

\par
\medskip
Finally, the last Section 4 is devoted to the precise study of the absolute continuity of the random nodal volume. In Section 4.1, we give a self-contained introduction to Malliavin calculus and Malliavin derivatives and in Section 4.2, we state and prove our main results concerning absolute continuity. 
We establish that the nodal volume of a suitable Gaussian field is in the domain of the Malliavin derivative (see conclusion (i) below) and that its distribution is not singular with respect to the Lebesgue measure. For more precise explanations about the domain $\mathbb{D}^{1,p}$ of the Malliavin derivative which appears in the statement (i) below,  we refer the reader to the subsection \ref{Malliavin-introduction}.

\begin{thm}\label{Main theorem-periodic}
Let $(f(x))_{x\in\mathbb{T}^d}$ be a periodic and stationary Gaussian field which is of class $C^2(\mathbb{T}^d,\R)$. We assume that the $(d+1)-$dimensional Gaussian vector $\left(f(x),\nabla f(x)\right)$ has a density.
Then, we have the two conclusions
\begin{itemize}
\item[(i)] $\mathcal{H}^{d-1}\left(\{f=0\}\right) \in\mathbb{D}^{1,\frac{d+1}{3}-}$,
\item[(ii)] provided that $\mathcal{H}^{d-1}\left(\{f=0\}\right)$ is not constant, its distribution has a non zero component which is absolutely continuous with respect to the Lebesgue measure.
\end{itemize}
\end{thm}

The previous statement also holds true without periodicity conditions on the Gaussian field $f$, as illustrated by our next main result. Let us stress here that we have separated the periodic and non-periodic frameworks because the absence of boundary terms in the integrations by parts used in the proof of Theorem \ref{Main theorem-periodic} makes it more transparent in the periodic setting. However, the strategy is essentially the same in both cases.

\begin{thm}\label{Mainpasperio}
Let $(f(x))_{x\in\mathbb{R}^d}$ be a stationary Gaussian field which is of class $C^2(\mathbb{R}^d,\R)$. Let $a_1<b_1, a_2<b_2,\cdots, a_d<b_d$ be some real numbers and set $A=\prod_{i=1}^d [a_i,b_i]$. We assume that the $(d+1)-$dimensional Gaussian vector $\left(f(x),\nabla f(x)\right)$ has a density.
Then we have the two conclusions
\begin{itemize}
\item[(i)] $\mathcal{H}^{d-1}\left(\{f=0\right\}\cap A) \in\mathbb{D}^{1,\frac{d+1}{3}-}$,
\item[(ii)] provided that $\mathcal{H}^{d-1}\left(\{f=0\}\cap A\right)$ is not constant, its distribution has a non zero component which is absolutely continuous with respect to the Lebesgue measure.
\end{itemize}
\end{thm}

\begin{rem}
Throughout the whole article, unless otherwise stated, we will assume that the considered Gaussian fields $f$, which in our case will be indexed by $\R^d$ or $\mathbb T^d$, are almost surely non-degenerated in the sense of \eqref{eq.non.deg}, i.e. almost surely no zeros of $f$ are critical points. When dealing with a stationary Gaussian field, the non degeneracy holds true if the $(d+1)-$dimensional Gaussian vector $(f(x),\nabla f (x))$ has a density with respect to Lebesgue measure, hence the assumption in the last two theorems. We refer the reader to Proposition 6.12 of \cite{azais2009level} for a proof and several related criteria. 
\end{rem}

\newpage

\begin{rem}\label{diracenzero}
Self-evidently, if a function $f$ has a constant sign on some domain $K$, its nodal set has volume zero. Hence, if $f$ is a Gaussian process which has positive probability to keep a constant sign on $K$, the distribution of the random variable $\mathcal{H}^{d-1}\left(\{ f=0\} \cap K\right)$ has an atom at zero and cannot be absolutely continuous with respect to the Lebesgue measure. In this sense, the conclusions (ii) of the two above Theorems are sharp. To illustrate this observation, let us consider $P_{\lambda}$ the random trigonometric polynomial of degree $3\times 3$ in $\mathbb T^3$ of the form 
\[
P_{\lambda}(x,y,z):=\lambda a_0 + \frac{1}{\sqrt{n}^3} \sum_{k,\ell,m=1}^n a_{k,\ell,m} \cos( kx)\cos(\ell y)\cos(m z), \quad (x,y,z) \in \mathbb T^3,
\]
where $a_0$ and the $a_{k,\ell,m}$ are independant standard Gaussian variables, and where $\lambda>0$ is a positive parameter. When $\lambda$ is zero, the random polynomial is spatially centered and necessarily vanishes. The nodal volume is then expected to be fully absolutely continuous which is consistent with the first picture below. When $\lambda$ gets higher, the probability for the Gaussian field to keep a constant sign increases and a Dirac mass appears. The latter is illustrated in red.
\begin{figure}[ht]
\begin{center}
\includegraphics[scale=0.6]{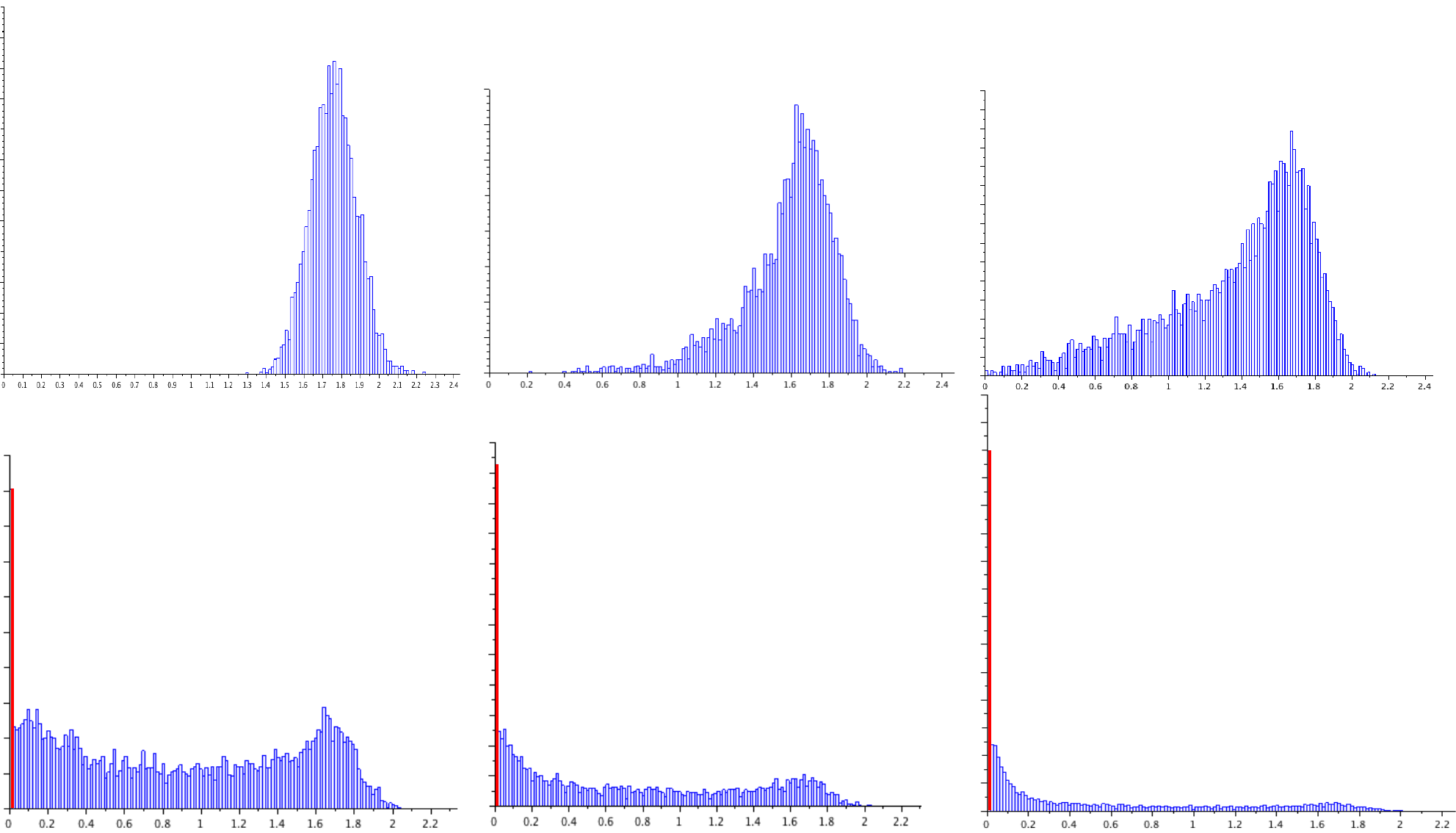}
\end{center}
\caption{Empirical histrograms (based on Monte-Carlo Method and the integral representation of Proposition \ref{formule-IPP1} below) of the nodal volume associated with the random trigonometric polynomial $P_{\lambda}$ for the different choices $\lambda = 0, 0.3, 0.5, 0.6,0.7,1$ (from left to right, top to bottom).}
\end{figure}
\end{rem}

\begin{rem}
Belonging to the domain of the Malliavin derivative is a true wealth of informations concerning the distribution. In addition to the celebrated Bouleau-Hirsch criterion which provides conditions ensuring the existence of densities, it is proved in \cite[prop 2.1.7, page 106]{nualart2006malliavin} that a random vector whose components belongs to the domain of the Mallavin derivative has a distribution whose topological support is connected. In the previous diagrams, the nodal volume appears to be supported on intervals of the form $[0,M]$ if $\lambda>0$ and of the form $[M_1,M_2]$ with $0<M_1< M_2$ when $\lambda=0$.
\end{rem}

\section{Closed Kac--Rice type formulas} \label{sec.Kac}
In this Section, starting from the classical Kac--Rice formula and using simple integrations by parts, we establish some exact closed formulas for the nodal volume associated with non-degenerated functions. 
\subsection{A closed formula in dimension one}
Let us first consider a periodic function $f\in\mathcal{C}^2(\mathbb{T},\R)$ which is supposed to be non-degenerated in the sense \eqref{eq.non.deg}.  In this one dimensional and periodic setting, the celebrated Kac--Rice formula \eqref{Kac--Rice-Intro} for the number of zeros of $f$ in $[0, 2\pi]$ simply reads
\begin{equation}\label{Kac--Rice}
 \mathcal H^{0}\left(\{f=0\}\right) =\lim_{\varepsilon \to 0} \int_0^{2\pi} \mathds{1}_{[-\varepsilon,\varepsilon]}(f(x)) |f'(x)| \frac{dx}{2\varepsilon}.
\end{equation}
The presence of the limit in $\varepsilon$ in the formula (\ref{Kac--Rice}) is an important drawback since one needs to track the speed of convergence in the Kac--Rice in order to get accurate estimates of the number of roots. We shall remove the limit and write the number of roots as a simple integral of an explicit functional of $(f,f',f'')$. The price to pay is to require two derivatives whereas the formula (\ref{Kac--Rice}) only needs one.
\begin{prop}\label{Kac--Rice-closed}If $f\in\mathcal{C}^2(\mathbb{T},\R)$ is non-degenerated, then we have
\[
\begin{array}{l}
 \mathcal H^{0}\left(\{f=0\}\right)  \displaystyle{= -\frac{1}{2}\int_0^{2\pi} \left( f''(x)f(x)-f'(x)^2 \right) \frac{|f(x)|}{\eta_f^3(x)}dx }.
\end{array}
\]
\end{prop}
\begin{proof}
By hypothesis, since $f$ is non-degenerated, we have $\eta_f:=\inf_{x \in [0, 2\pi]} \eta_f(x)>0$, so that we can write
\[
\begin{array}{ll}
\displaystyle{\frac{|f'(x)|^2}{\sqrt{|f(x)|^2 + |f'(x)|^2}} - |f'(x)|} & =\displaystyle{ \sqrt{|f(x)|^2 + |f'(x)|^2}  - |f'(x)| - \frac{|f(x)|^2}{\sqrt{|f(x)|^2 + |f'(x)|^2}}}\\
\\
& = \displaystyle{\frac{|f(x)|^2}{\sqrt{|f(x)|^2 + |f'(x)|^2} + |f'(x)|}-\frac{|f(x)|^2}{\sqrt{|f(x)|^2 + |f'(x)|^2}}},
\end{array}
\]
and in particular, we get 
\[
\left| \frac{|f'(x)|^2}{\sqrt{|f(x)|^2 + |f'(x)|^2}} - |f'(x)| \right| \leq \frac{2 f(x)^2}{\eta_f}.
\]
For all $\varepsilon>0$, if we set
\[
I_{\varepsilon}  \displaystyle{:=\int_0^{2\pi} \mathds{1}_{[-\varepsilon,\varepsilon]}(f(x)) \frac{|f'(x)|^2}{\sqrt{|f(x)|^2 + |f'(x)|^2}} \frac{dx}{2\varepsilon} - \int_0^{2\pi} \mathds{1}_{[-\varepsilon,\varepsilon]}(f(x)) |f'(x)| \frac{dx}{2\varepsilon}},
\]
we have then
\[
\begin{array}{ll}
| I_{\varepsilon} | & \leq  \displaystyle{ \int_0^{2\pi} \mathds{1}_{[-\varepsilon,\varepsilon]}(f(x))  \frac{2 f(x)^2}{\eta_f} \frac{dx}{2\varepsilon} \leq \frac{2\pi  \varepsilon}{\eta_f}}.
\end{array}
\]
In particular, $\lim_{\varepsilon \to 0} I_{\varepsilon} =0$ so that the Kac--Rice formula \eqref{Kac--Rice} can be rewritten as 
\[
 \mathcal H^{0}\left(\{f=0\}\right)=\lim_{\varepsilon \to 0} N_{\varepsilon}, \quad \text{where} \;\; N_{\varepsilon}:=\int_0^{2\pi} \mathds{1}_{[-\varepsilon,\varepsilon]}(f(x)) \frac{|f'(x)|^2}{\sqrt{|f(x)|^2 + |f'(x)|^2}}  \frac{dx}{2\varepsilon}.
\]
For all $\varepsilon>0$, let us now consider the function $\phi_{\varepsilon}$ defined as
\[
\phi_{\varepsilon}(x) := \left \lbrace \begin{array}{cl} -1 & \text{if} \;\; x \leq -\varepsilon, \\
x/\varepsilon & \text{if} \;\; -\varepsilon \leq x \leq \varepsilon, \\
1 & \text{if} \;\;  x \geq \varepsilon. \end{array} \right.
\]
Since $f$ is periodic, integrating by parts, we get that
\[
\begin{array}{ll}
N_{\varepsilon}& = \displaystyle{ \frac{1}{2} \int_0^{2\pi} \left( \frac{1}{\varepsilon}\mathds{1}_{[-\varepsilon,\varepsilon]}(f(x)) f'(x)\right) \frac{f'(x)}{\sqrt{|f(x)|^2 + |f'(x)|^2}}  dx}\\
\\
&= \displaystyle{ \frac{1}{2} \int_0^{2\pi} \left( \phi_{\epsilon} \circ f(x) \right)' \frac{f'(x)}{\sqrt{|f(x)|^2 + |f'(x)|^2}}  dx}\\
\\
&= \displaystyle{ -\frac{1}{2} \int_0^{2\pi} \left( \phi_{\epsilon} \circ f(x) \right) \left(\frac{f'(x)}{\sqrt{|f(x)|^2 + |f'(x)|^2}} \right)' dx}\\
\\
&= \displaystyle{ -\frac{1}{2} \int_0^{2\pi} \left( \phi_{\epsilon} \circ f(x) \right) \frac{f(x) \left( f''(x)f(x) - f'(x)^2\right)}{\eta_f(x)^3}  dx}.
\end{array}
\]
Since $|\phi_{\varepsilon}|$ is bounded by one uniformly in $\varepsilon$, by the dominated convergence Theorem, we deduce that
\[
\begin{array}{ll}
 \mathcal H^{0}\left(\{f=0\}\right) & =\displaystyle{\lim_{\varepsilon \to 0} N_{\varepsilon} = -\frac{1}{2} \int_0^{2\pi}  \text{sign}( f(x))  \frac{f(x) \left( f''(x)f(x) - f'(x)^2\right)}{\eta_f(x)^3}  dx}  \\
\\
& =\displaystyle{ -\frac{1}{2} \int_0^{2\pi}     \left( f''(x)f(x) - f'(x)^2\right)\frac{|f(x)|}{\eta_f(x)^3}  dx.}  
\end{array}
\]
\end{proof}


\subsection{Understanding the formula and generalizations}
In this subsection, we shall observe that the previous procedure hides a simple phenomenon which allows us to derive a family of analogue formulas. Let $F$ be a $C^1$ function on $\mathbb R$ such that 
\[
\lim_{x \to -\infty} F(x) =-1, \quad \lim_{x \to +\infty} F(x) =1.
\]
The reader can keep in mind the examples 
\[
F(x) = \frac{x}{\sqrt{1+x^2}}, \quad F(x) =\frac{2}{\pi} \arctan(x), \quad \text{or} \quad F=2 G -1,
\]
where $G$ is the cumulative distribution function of any continuous random variable. 
 
\begin{prop}\label{Kac--Rice-general}If $f\in\mathcal{C}^2(\mathbb{T},\R)$  is non-degenerated, then we have
\begin{equation}\label{eq.genF}
\begin{array}{ll}
 \mathcal H^{0}\left(\{f=0\}\right) & \displaystyle{= -\frac{1}{2}\int_0^{2\pi} F'\left( \frac{f'(x)}{f(x)}\right) \left(\frac{f'(x)}{f(x)}\right)'  dx}.
\end{array}
\end{equation}
\end{prop}

\begin{proof}
Let us first remark that since $f$ is non-degenerated, it has a finite number of zeros in $[0,2\pi]$. Moreover, if $f$ does not vanish, then the integrand in the right hand side of Equation \eqref{eq.genF} has no singularity and the integral is indeed equal to zero by periodicity. To simplify the expressions, set $N:= \mathcal H^{0}\left(\{f=0\}\right)$ and denote by $x_1\leq \ldots \leq x_N$ the zeros of $f$ in $[0, 2\pi]$, and set $x_{N+1}:=x_1$. We can then decompose the integral in Equation \eqref{eq.genF} as  the sum
\[
\begin{array}{l}
\displaystyle{ -\frac{1}{2}\int_0^{2\pi} F'\left( \frac{f'(x)}{f(x)}\right) \left(\frac{f'(x)}{f(x)}\right)'  dx=-\frac{1}{2} \sum_{i=1}^N \int_{x_i^+}^{x_{i+1}^-} F'\left( \frac{f'(x)}{f(x)}\right) \left(\frac{f'(x)}{f(x)}\right)'  dx} \\
\\
=  - \displaystyle{\frac{1}{2}\sum_{i=1}^N \underbrace{ F\left( \frac{f'}{f}\left(x_{i}^-\right)\right) -F\left( \frac{f'}{f}\left(x_{i}^+\right)\right)}_{=-2}=N. }
\end{array}
\]
Indeed, as illustrated Figure 2 below, each zero crossing contributes to a factor $-2$, since $F$ takes values $\pm 1$ at $\pm \infty$.
\end{proof}
\begin{figure}[ht]\label{fig.logd}
\vspace{-0.5cm}\begin{center}\scalebox{0.8}{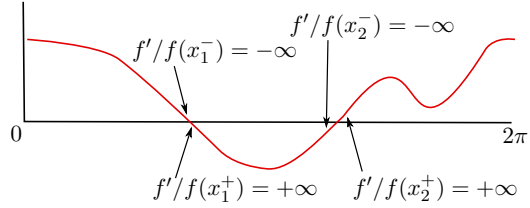}\end{center}
\caption{Values of the logarithmic derivative at the successive zero crossings.}
\end{figure}

\begin{exa}\label{argsh}
If we make the choice 
\[
F(x) = \frac{x}{\sqrt{1+x^2}}, \quad \text{then} \quad F'(x) = \frac{1}{\sqrt{1+x^2}^3},
\]
and the formula established in Proposition \ref{Kac--Rice-general} reads
\[
\begin{array}{ll}
 \mathcal H^{0}\left(\{f=0\}\right) & =\displaystyle{-\frac{1}{2}\int_0^{2\pi} F'\left( \frac{f'(x)}{f(x)}\right) \left(\frac{f'(x)}{f(x)}\right)'  dx }\\
 
&= \displaystyle{-\frac{1}{2}\int_0^{2\pi} \frac{|f(x)|}{\sqrt{f^2(x)+ f'(x)^2}^3} \times \left(f''(x)f(x)-f'(x)^2\right) dx},
\end{array}
\]
\if{\[
\begin{array}{ll}
 \mathcal H^{0}\left(\{f=0\}\right) & =\displaystyle{-\frac{1}{2}\int_0^{2\pi} F'\left( \frac{f'(x)}{f(x)}\right) \left(\frac{f'(x)}{f(x)}\right)'  dx }\\
\\
& =\displaystyle{-\frac{1}{2}\int_0^{2\pi} \frac{1}{\sqrt{1+\left( \frac{f'(x)}{f(x)}\right)^2}^3} \times \left(\frac{f'(x)}{f(x)}\right)'  dx}\\
\\
&= \displaystyle{-\frac{1}{2}\int_0^{2\pi} \frac{1}{\sqrt{1+\left( \frac{f'(x)}{f(x)}\right)^2}^3} \times \left(\frac{f''(x)f(x)-f'(x)^2}{f(x)^2}\right) dx}
\\
\\
&= \displaystyle{-\frac{1}{2}\int_0^{2\pi} \frac{|f(x)|}{\sqrt{f^2(x)+ f'(x)^2}^3} \times \left(f''(x)f(x)-f'(x)^2\right) dx},
\end{array}
\]}\fi
that is, we recover the formula of Proposition \ref{Kac--Rice-closed}. 
\end{exa}

\begin{exa}\label{arctan}
If we now choose 
\[
F(x) = \frac{2}{\pi}\arctan(x), \quad \text{then} \quad F'(x) = \frac{2}{\pi}\frac{1}{1+x^2},
\]
so that 
\[
\begin{array}{ll}
 \mathcal H^{0}\left(\{f=0\}\right) & =\displaystyle{-\frac{1}{2}\int_0^{2\pi} F'\left( \frac{f'(x)}{f(x)}\right) \left(\frac{f'(x)}{f(x)}\right)'  dx }\\
\\
&= \displaystyle{-\frac{1}{\pi}\int_0^{2\pi} \frac{1}{f(x)^2+ f'(x)^2} \times \left(f''(x)f(x)-f'(x)^2\right) dx}.
\end{array}
\]
\if{\[
\begin{array}{ll}
 \mathcal H^{0}\left(\{f=0\}\right) & =\displaystyle{-\frac{1}{2}\int_0^{2\pi} F'\left( \frac{f'(x)}{f(x)}\right) \left(\frac{f'(x)}{f(x)}\right)'  dx }\\
\\
& =\displaystyle{-\frac{1}{\pi}\int_0^{2\pi} \frac{1}{1+\left( \frac{f'(x)}{f(x)}\right)^2} \times \left(\frac{f'(x)}{f(x)}\right)'  dx}\\
\\
&= \displaystyle{-\frac{1}{\pi}\int_0^{2\pi} \frac{1}{1+\left( \frac{f'(x)}{f(x)}\right)^2}\times \left(\frac{f''(x)f(x)-f'(x)^2}{f(x)^2}\right) dx}
\\
\\
&= \displaystyle{-\frac{1}{\pi}\int_0^{2\pi} \frac{1}{f(x)^2+ f'(x)^2} \times \left(f''(x)f(x)-f'(x)^2\right) dx}.
\end{array}
\]}\fi
\end{exa}

\begin{exa}\label{indicator}
But one can take more degenerate examples for instance $F=2G-1$ where $G$ is the cumulative distribution of the uniform distirbution on $[-1,1]$. This leads to
\[
\begin{array}{ll}
 \mathcal H^{0}\left(\{f=0\}\right) & =\displaystyle{-\frac{1}{2}\int_0^{2\pi} F'\left( \frac{f'(x)}{f(x)}\right) \left(\frac{f'(x)}{f(x)}\right)'  dx }\\
\\
& =\displaystyle{-\frac{1}{2}\int_0^{2\pi} \textbf{1}_{\left\{\left|\frac{f'(x)}{f(x)}\right|\le 1\right\}}\times \left(\frac{f'(x)}{f(x)}\right)'  dx}\\
\\
& =\displaystyle{-\frac{1}{2}\int_0^{2\pi} \textbf{1}_{\left\{\left|\frac{f'(x)}{f(x)}\right|\le 1\right\}}\times\frac{f''(x)f(x)-f'(x)^2}{f(x)^2}  dx}.
\end{array}\]
\end{exa}

Going back to Example 2, we have therefore proved the following simple formula which has the advantage to express the number of roots as an integral of a simple rational function of $(f,f',f'')$. It might be of particular interest when dealing with analytic functions since the integrand in the equation (\ref{crazy-formula}) remains analytic even in degenerate settings, id-est when $\eta_f$ vanishes.

\begin{coro}\label{cor.crazy-formula}
If $f\in\mathcal{C}^2(\mathbb{T},\R)$  is non-degenerated, then we have
\begin{equation}\label{crazy-formula}
 \mathcal H^{0}\left(\{f=0\}\right) =\displaystyle{\frac{1}{\pi} \int_0^{2\pi} \frac{f'(x)^2-f(x) f''(x)}{f^2(x)+f'(x)^2} }dx.
\end{equation}
\end{coro}

On the other hand, such formulas can also be used in order to bound the number of real roots in term of various reformulations of the non degeneracy assumption that $f$ and $f'$ do not vanish simultaneously. One gets for instance the following estimates,
\begin{coro}\label{bound-number-roots}
Using the formula given in Example \ref{arctan} on gets
\[
 \mathcal H^{0}\left(\{f=0\}\right) \le \frac{1}{\pi} \int_0^{2\pi}\frac{|f''(x)|}{\eta_f(x)}dx+2 \leq 2\left( \frac{||f''||_{\infty}}{\eta_f} +1\right).
 \]
The formula given in Example \ref{indicator} gives instead
\[
 \mathcal H^{0}\left(\{f=0\}\right) \le \frac{1}{2}\int_0^{2\pi}\textbf{1}_{\left\{\left|\frac{f'}{f}\right|\le 1\right\}} \frac{|f''(x)|}{|f(x)|} dx+\pi.
\]
\end{coro}

\subsection{Extension to the non-periodic setting}

Whereas it is particularly convenient to deal with periodic functions as boundary terms vanish when performing integrations by parts, the above approach involving the logarithmic derivative of $f$ allows also to deal with non periodic functions on any interval. Namely, we have

\begin{prop}\label{gene-formula}Let $f$ be non-degenerated $C^2$ function on $\mathbb R$ and let $F$ be a $C^1$ function on $\mathbb R$ such that 
\[
\lim_{x \to -\infty} F(x) =-1, \quad \lim_{x \to +\infty} F(x) =1.
\] 
Fix $a<b$ and suppose for simplicity that $f(a)f(b)\neq 0$. Then, we have the formula
\[
\begin{array}{ll}
\mathcal H^0\left(\{f=0\} \cap [a,b]  \right)  & \displaystyle{= \frac{1}{2}\left[F\left( \frac{f'}{f}(b)\right)-F\left( \frac{f'}{f}(a)\right)- \int_a^b F'\left( \frac{f'(x)}{f(x)}\right) \left(\frac{f'(x)}{f(x)}\right)'  dx\right]}.
\end{array}
\]
In particular, if $a$ and $b$ are the loci of local extrema of $f$, then we have 
\[
\begin{array}{ll}
\mathcal H^0\left(\{f=0\} \cap [a,b]  \right)  & \displaystyle{= -\frac{1}{2} \int_a^b F'\left( \frac{f'(x)}{f(x)}\right) \left(\frac{f'(x)}{f(x)}\right)'  dx}.
\end{array}
\]
\end{prop}

\begin{proof}
As above, let us denote by $x_1 \leq \ldots \leq x_N$ the consecutive zeros of $f$ in $[a, b]$. We can decompose the integral as 
\[
\begin{array}{ll}
I & :=\displaystyle{ \int_a^{b} F'\left( \frac{f'(x)}{f(x)}\right) \left(\frac{f'(x)}{f(x)}\right)'  dx }\\
\\
& \displaystyle{=\int_{a}^{x_{1}^-} F'\left( \frac{f'(x)}{f(x)}\right) \left(\frac{f'(x)}{f(x)}\right)'  dx + \int_{x_N^+}^{b} F'\left( \frac{f'(x)}{f(x)}\right) \left(\frac{f'(x)}{f(x)}\right)'  dx} \\
\\
& +\displaystyle{ \sum_{i=1}^{N-1} \int_{x_i^+}^{x_{i+1}^-} F'\left( \frac{f'(x)}{f(x)}\right) \left(\frac{f'(x)}{f(x)}\right)'  dx,} 
\end{array}
\]
i.e.
\[
\begin{array}{ll}
I & =  \displaystyle{F\left( \frac{f'}{f}(b)\right)-F\left( \frac{f'}{f}(a)\right) + \sum_{i=1}^N \underbrace{ F\left( \frac{f'}{f}\left(x_{i}^-\right)\right) -F\left( \frac{f'}{f}\left(x_{i}^+\right)\right)}_{=-2}. }\\
\\
& =  \displaystyle{F\left( \frac{f'}{f}(b)\right)-F\left( \frac{f'}{f}(a)\right)  -2\times N,}
\end{array}
\]
hence the first formula. If $a$ and $b$ are the loci of local extrema of $f$, we have $f(a)f(b)\neq 0$ (since $f$ is non-degenerated) and $f'(a)=f'(b)=0$ so that $f'/f(a)=f'/f(b)=0$ and the boundary terms vanish, hence the second formula.
\end{proof}

Let us now generalize the previous results to functions which possibly exhibit double, or higher order zeros. The next Proposition shows that the formula obtained in Proposition \ref{gene-formula} actually holds for functions with non-flat zeros, in particular for the large class of quasi-analytic functions. We stress that the forthcoming formulas hold for the number of roots \textit{without counting multiplicity}.

\begin{prop}Let $f$ be a $C^{\infty}$ function on $\mathbb R$ which does not have flat zero, namely if $f(x)=0$ for $x \in \mathbb R$, then there exists $r \in \mathbb N$ such that $f^{(r)}(x) \neq 0$. Let $F$ be a $C^1$ function on $\mathbb R$ such that 
\[
\lim_{x \to -\infty} F(x) =-1, \quad \lim_{x \to +\infty} F(x) =1.
\] 
Fix $a<b$ such that $f(a)f(b)\neq 0$. Then the number of zeros of $f$ in $[a,b]$ is finite and it is given by the formula
\[
\mathcal H^0\left(\{f=0\} \cap [a,b]  \right) = \frac{1}{2}\left[F\left( \frac{f'}{f}(b)\right)-F\left( \frac{f'}{f}(a)\right)- \int_a^b F'\left( \frac{f'(x)}{f(x)}\right) \left(\frac{f'(x)}{f(x)}\right)'  dx\right].
\]
\label{prop.weak.deg}
\end{prop}

\begin{proof}
As in the proof of Proposition \ref{Kac--Rice-general} above, let us first note that if $f$ does not vanish on $[a,b]$, then the right hand side of the last equation vanishes since the integrand is smooth on the whole interval. Now, if $f(x)=0$ and $r:=\inf_{k\in \mathbb N}\{k,\, f^{(k)}(x)\neq 0\}$, performing a Taylor expansion, we have for $h$ small enough
\begin{eqnarray*}
f(x+h) &=&\frac{f^{(r)}(x) h^r}{r!}\left( 1 +o(1)\right),\\ f'(x+h) &=& \frac{f^{(r)}(x) h^{r-1}}{(r-1)!}\left( 1 +o(1)\right),\\
f''(x+h) &=& \frac{f^{(r)}(x) h^{r-2}}{(r-2)!}\left( 1 +o(1)\right).
\end{eqnarray*}
In particular, $f$ has only one zero in a small open neighborhood $V_x$ of $x$. Being compact in $[a, b]$, the nodal set of $f$ can be covered by a finite number of the $V_x$, hence it is finite.
As above, let us denote by $x_1 \leq \ldots \leq x_N$ the consecutive zeros of $f$ in $[a, b]$. If $x$ is one of these zeros, we have for $h$ small enough
\[
\frac{ f'(x+h)}{ f(x+h)} = \frac{r}{h} \left( 1 +o(1)\right), \quad \text{and thus} \;\; \lim_{h \to 0^{\pm}} \frac{ f'(x+h)}{ f(x+h)} =\pm \infty.
\]
Moreover, for $h$ small enough we have also
\[
\frac{ f''(x+h)f(x+h)-f'^2(x+h)}{ f^2(x+h)+f'^2(x+h)} = -\frac{1}{r}\times \frac{1}{1+\frac{h^2}{r^2}} \left( 1 +o(1)\right),
\]
so that the ratio is locally bounded and hence locally integrable. As in the proof of Proposition \ref{gene-formula}, we can thus decompose the integral
\begin{eqnarray*}
I &:=& \int_a^{b} F'\left( \frac{f'(x)}{f(x)}\right) \left(\frac{f'(x)}{f(x)}\right)'  dx \\
&=& F\left( \frac{f'}{f}(b)\right)-F\left( \frac{f'}{f}(a)\right) + \sum_{i=1}^N \underbrace{ F\left( \frac{f'}{f}\left(x_{i}^-\right)\right) -F\left( \frac{f'}{f}\left(x_{i}^+\right)\right)}_{=-2} \\
& =&  F\left( \frac{f'}{f}(b)\right)-F\left( \frac{f'}{f}(a)\right)  -2\times N,
\end{eqnarray*}
hence the result.
\end{proof}
Applying Proposition \ref{prop.weak.deg} with the choice of counting function $F=\frac{2}{\pi} \arctan$, we thus get the analogue of formula \eqref{crazy-formula} of  Corollary \ref{cor.crazy-formula}  for non-periodic functions with non-flat zeros, namely

\begin{coro}
Let $f$ be a $C^{\infty}$ function on $\mathbb R$ which does not have flat zero, and let $a<b$ such that $f(a)f(b)\neq 0$. Then the number of zeros of $f$ in $[a,b]$ is finite and it is given by the formula
\[
\mathcal H^0\left(\{f=0\} \cap [a,b]  \right) =\frac{1}{\pi} \left[ \arctan \frac{f'(b)}{f(b)}-\arctan \frac{f'(a)}{f(a)} +   \int_a^b \frac{f'^2(x)-f(x)f''(x)}{f^2(x)+f'^2(x)}dx\right].
\]
\end{coro}

\begin{rem}
Note that the hypotheses on the auxiliary function $F$ imply that it is bounded so that the boundary terms in Proposition \ref{gene-formula} and \ref{prop.weak.deg} are bounded. Therefore, these boundary terms are not annoying if we have in mind some applications where the number of zeros becomes large in a certain regime. We will exploit this information in the section \ref{birk-section}.
\end{rem}

\subsection{Closed formulas in higher dimensions}

Using the exact same approach as above, i.e. using simple integrations by parts starting from the standard Kac--Rice formula \eqref{Kac--Rice-Intro},  let us now exhibit analogue closed Kac--Rice type formulas in a higher dimensional framework. 

\begin{prop}\label{formule-IPP1} Let $f$ be a $C^2$ periodic function on $\mathbb T^d$ which is non-degenerated.
Then the volume of the total nodal set is given by
\[
\begin{array}{l}
\mathcal H^{d-1}\left(\{f=0\}\right)  =\displaystyle{ -\frac{1}{2} \int_{\mathbb T^d}  \left( f(x) \Delta f(x)-  ||\nabla f(x)||^2 \right) \frac{|f(x)|}{\eta_f(x)^3} dx}\\
\\
\displaystyle{ - \frac{1}{2} \int_{\mathbb T^d} \mathrm{sign}(f(x))  \left(\Delta f(x) ||\nabla f(x)||^2 - \nabla f(x)^* \mathrm{Hess}_x(f) \nabla f(x)\right)\frac{dx}{\eta_f(x)^3}}.
\end{array}
\]
\end{prop}

\begin{rem}
Note that in dimension $d=1$, the second term on the right hand side vanishes so that the expression is consistent with the one given in Proposition 1. Moreover, the formula is homogeneous i.e. invariant under $f \leftrightarrow \lambda f$, as it should.
\end{rem}
\begin{proof}
The proof follows the same lines as its one dimensional analogue. Namely, as in dimension one, the classical Kac--Rice formula 
\[
\mathcal H^{d-1}\left( \{f =0\} \right) = \lim_{\varepsilon \to 0} \int_{\mathbb T^d} \mathds{1}_{[-\varepsilon,\varepsilon]}(f(x)) ||\nabla f(x)|| \frac{dx}{2\varepsilon},
\]
can be rewritten as 
\[
\mathcal H^{d-1}\left( \{f =0\} \right) = \lim_{\varepsilon \to 0} V_{\varepsilon}, \;\; \text{where} \;\; V_{\varepsilon}:=\int_{\mathbb T^d} \mathds{1}_{[-\varepsilon,\varepsilon]}(f(x)) \frac{||\nabla f(x)||^2}{\sqrt{|f(x)|^2 + ||\nabla f||^2}} \frac{dx}{2\varepsilon}.
\]
We have then 
\[
V_{\varepsilon} = \sum_{i=1}^d V_{\varepsilon}^i, \;\; \text{where} \;\; V_{\varepsilon}^i:= \int_{\mathbb T^d} \mathds{1}_{[-\varepsilon,\varepsilon]}(f(x)) \frac{|\partial_i f(x)|^2}{\sqrt{|f(x)|^2 + ||\nabla f||^2}} \frac{dx}{2\varepsilon}.
\]
As in dimension one, an integration by parts in the $i^{th}$ variable yields
\[
\begin{array}{ll}
V_{\varepsilon}^i & = \displaystyle{ \frac{1}{2} \int_{\mathbb T^d} \partial_i \left( \phi_{\epsilon} \circ f(x) \right) \left(  \frac{\partial_i f(x)}{\sqrt{|f(x)|^2 + ||\nabla f||^2}}\right)  dx}\\
\\
& =\displaystyle{ - \frac{1}{2} \int_{\mathbb T^d} \left( \phi_{\epsilon} \circ f(x) \right) \partial_i \left(  \frac{\partial_i f(x)}{\sqrt{|f(x)|^2 + ||\nabla f||^2}}\right)  dx}\\
\\
& =\displaystyle{ - \frac{1}{2} \int_{\mathbb T^d} \left( \phi_{\epsilon} \circ f(x) \right) \left(  \frac{f^2\partial_{ii}^2 f - f(\partial_i f)^2+\partial_{ii}^2 f ||\nabla f||^2-\partial_i f \sum_j \partial_j f \partial_{ij}^2f}{\eta_f(x)^3}\right)  dx},\\
\end{array}
\]
where $\phi_\epsilon$ is the function introduced in the proof of Proposition \ref{Kac--Rice-closed}. We deduce that 
\[
\begin{array}{ll}
V_{\varepsilon} & =\displaystyle{ - \frac{1}{2} \int_{\mathbb T^d} \left( \phi_{\epsilon} \circ f(x) \right) \left( f^2(x) \Delta f(x)- f(x) ||\nabla f(x)||^2 \right) \frac{dx}{\eta_f(x)^3} }\\
\\
&  \displaystyle{- \frac{1}{2} \int_{\mathbb T^d} \left( \phi_{\epsilon} \circ f(x) \right)  \left(\Delta f(x) ||\nabla f(x)||^2 - \nabla f(x)^* \text{Hess}_x(f) \nabla f(x)\right)\frac{dx}{\eta_f(x)^3}}.
\end{array}
\]
Letting $\varepsilon$ go to zero, one deduces that 
\[
\begin{array}{ll}
&\mathcal H^{d-1}\left(\{f=0\}\right) =\displaystyle{ - \frac{1}{2} \int_{\mathbb T^d}  \left( f(x) \Delta f(x)-  ||\nabla f(x)||^2 \right) \frac{|f(x)|}{\eta_f(x)^3} dx}\\
\\
&  \displaystyle{- \frac{1}{2} \int_{\mathbb T^d} \text{sign}(f(x))  \left(\Delta f(x) ||\nabla f(x)||^2 - \nabla f(x)^* \text{Hess}_x(f) \nabla f(x)\right)\frac{dx}{\eta_f(x)^3}}.
\end{array}
\]
\end{proof}
With a slight variation of the proof, it is possible to obtain a more concise expression.
\begin{prop}\label{formulecompacte}
Let $f$ be a $C^2$ periodic function on $\mathbb T^d$ which is non-degenerated.
Then the volume of the total nodal set is given by
\[
\begin{array}{ll}
\mathcal H^{d-1}\left(\{f=0\}\right) 
& \displaystyle{= -\frac{1}{2} \int_{\mathbb T^d} \mathrm{sign}(f(x)) \times \Delta\left( \frac{ f(x)}{ \eta_f(x)}\right) dx}.
\end{array}
\]
\end{prop}
 \begin{proof}
As in the proof of Proposition 5, we have
 \[
 \begin{array}{ll}
\mathcal H^{d-1}\left( \{f =0\} \right) & = \displaystyle{\lim_{\varepsilon \to 0} \int_{\mathbb T^d} \mathds{1}_{[-\varepsilon,\varepsilon]}(f(x)) \frac{||\nabla f(x)||^2}{\sqrt{|f(x)|^2 + ||\nabla f||^2}} \frac{dx}{2\varepsilon}}.
\end{array}
\]
In other words, 
 \[
 \begin{array}{ll}
\mathcal H^{d-1}\left( \{f =0\} \right) & =  \displaystyle{\sum_{i=1}^d\lim_{\varepsilon \to 0} \int_{\mathbb T^d} \left( \frac{\mathds{1}_{[-\varepsilon,\varepsilon]}(f(x)) }{2\varepsilon}\partial_i f(x) \right) \frac{\partial_i f(x)}{\sqrt{|f(x)|^2 + ||\nabla f||^2}}dx} .
\end{array}
\]
Now, we have also
\[
\frac{\partial_i f(x)}{\eta_f(x)} = \partial_i \left(\frac{f(x)}{\eta_f(x)}\right)- f(x) \partial_i \left(\frac{1}{\eta_f(x)}\right).
\]
and since 
\[
\left[ \frac{\mathds{1}_{[-\varepsilon,\varepsilon]}(f(x)) }{2\varepsilon} \times f \right| \leq 1
\]
by dominated convergence, we deduce
\[
\lim_{\varepsilon \to 0} \int_{\mathbb T^d} \left( \frac{\mathds{1}_{[-\varepsilon,\varepsilon]}(f(x)) }{2\varepsilon}\partial_i f(x) \right) f(x) \partial_i  \left(\frac{1}{\eta_f(x)}\right)dx=0.
\]
We thus get
 \[
 \begin{array}{ll}
\mathcal H^{d-1}\left( \{f =0\} \right) & = \displaystyle{\sum_{i=1}^d\lim_{\varepsilon \to 0} \int_{\mathbb T^d} \left( \frac{\mathds{1}_{[-\varepsilon,\varepsilon]}(f(x)) }{2\varepsilon}\partial_i f(x) \right) \partial_i  \left(\frac{f(x)}{\eta_f(x)}\right)dx} .
\end{array}
\]
As in the proof of Proposition 5, an integration by parts associated with the dominated convergence theorem then yields the desired result.
 \end{proof}
 \begin{rem}
 Note that this last formula is compact but it has the disadvantage of having a $\mathrm{sign}(f(x))$ in it and Laplacian of the ratio implicitly involves third derivatives of $f$, which is not the case of the formula of Proposition 5.
 \end{rem}
 
 \subsection{A non singular formula for the nodal volume}
The main drawback of the closed formulas for the nodal volume obtained in  Proposition \ref{formule-IPP1} and \ref{formulecompacte} is the presence of the term $\text{sign}(f)$ which is not a Lipschitz functional of $f$. Indeed, we have in mind to use these formula associated with Malliavin calculus and the latter requires Lipschitz regularity. In order to bypass this problem, one needs to perform an additional integration by parts, which will require, in turn, three derivatives for $f$. Nevertheless, as we will see just below, the derivatives of order $3$ will cancel in the computations. As a result, without additional regularity assumptions, a less singular formula holds for the nodal volume, which is the content of the next Proposition.
 
\begin{prop}\label{prop.IPP2}
Let $f\in C^2(\mathbb{T}^d, \mathbb R)$ which is non-degenerated, then we have
\[
\begin{array}{ll}
\mathcal H^{d-1}\left(\{f=0\}\right) & =\displaystyle{ - \frac{1}{2} \int_{\mathbb T^d}  \left( f(x) \Delta f(x)-  ||\nabla f(x)||^2 \right) \frac{|f(x)|}{\eta_f(x)^3} dx}\\
\\
&  \displaystyle{+\int_{\mathbb{T}^d} |f(x)| \left(\|\text{Hess}_x (f)\|^2-\text{Tr}\left(\text{Hess}_x(f)\right)^2\right)\frac{dx}{\eta_f(x)^3}}\\
\\
& \displaystyle{+\frac{3}{2} \int_{\mathbb{T}^d}\frac{|f(x)|}{\eta_f^5(x)} \left(\Delta f(x) \langle \nabla f(x), \nabla \eta_f^2(x)\rangle -\nabla f(x)^* \mathrm{Hess}_x f \nabla \eta_f^2(x)\right) dx.}
\end{array}
\]
\end{prop}
\begin{proof} 
Our starting point is the formula established in Proposition \ref{formule-IPP1}. We focus on the singular term, i.e. the one containing $\mathrm{sign}(f)$, that we will call $A$. Let us first suppose that $f\in C^3(\mathbb{T}^d, \mathbb R)$ and perform the following integrations by parts.
\begin{eqnarray*}
A&:=&\int_{\mathbb T^d} \text{sign}(f(x))  \left(\Delta f(x) ||\nabla f(x)||^2 - \nabla f(x)^* \text{Hess}_x(f) \nabla f(x)\right)\frac{dx}{\eta_f(x)^3}\\
&=& \int_{\mathbb{T}^d} \Delta f(x) \nabla |f(x)| \cdot \nabla f(x) \frac{dx}{\eta_f(x)^3} - \frac{1}{2}\int_{\mathbb{T}^d} \nabla |f(x)| \cdot \nabla \|\nabla f(x)\|^2 \frac{dx}{\eta_f(x)^3}\\
&\stackrel{\text{I.B.P.}}{=}&-\int_{\mathbb{T}^d} \left(\Delta f(x)\right)^2 |f(x)| \frac{dx}{\eta_f(x)^3}-\int_{\mathbb{T}^d} |f(x)| \nabla f(x) \cdot \nabla \left(\frac{\Delta f(x)}{\eta_f(x)^3}\right) dx\\
&+& \frac{1}{2} \int_{\mathbb{T}^d} |f(x)| \Delta \left(\|\nabla f(x)\|^2 \right) \frac{d x}{\eta_f(x)^3}+\frac{1}{2} \int_{\mathbb{T^d}} |f(x)| \nabla \left(\|\nabla f(x)\|^2\right)\cdot \nabla \left(\frac{1}{\eta_f^3 (x)}\right) dx.
\end{eqnarray*}
Now we compute
\begin{eqnarray*}
\Delta \left(\|\nabla f(x)\|^2\right)&=&2 \sum_{i=1}^d \sum_{j=1}^d \partial^3_{i,i,j} f(x) \partial_{j} f(x)+2 \sum_{i,j=1}^d \left(\partial^2_{i,j}f(x)\right)^2\\
&=&2~ \nabla f(x) \cdot \nabla \left(\Delta f(x)\right)+2 \sum_{i,j=1}^d \left(\partial^2_{i,j}f(x)\right)^2.
\end{eqnarray*}
Substituting this equality in the previous equations, we get
\begin{eqnarray*}
A&=& -\int_{\mathbb{T}^d} \left(\Delta f(x)\right)^2 |f(x)| \frac{dx}{\eta_f(x)^3}-\int_{\mathbb{T}^d} |f(x)| \Delta f(x) \nabla f(x) \cdot \nabla \left(\frac{1}{\eta_f(x)^3}\right) dx\\
& & +\int_{\mathbb{T}^d}|f(x)| \sum_{i,j=1}^d \left(\partial^2_{i,j} f(x)\right)^2  \frac{dx}{\eta_f(x)^3}+     \frac{1}{2} \int_{\mathbb{T}^d} |f(x)| \nabla \left(\|\nabla f(x)\|^2\right)\cdot \nabla \left(\frac{1}{\eta_f^3 (x)}\right) dx\\
&=& \int_{\mathbb{T}^d} |f(x)| \left(\|\text{Hess}_x (f)\|^2-\text{Tr}\left(\text{Hess}_x(f)\right)^2\right)\frac{dx}{\eta_f(x)^3}\\
& & + \frac{3}{2}\int_{\mathbb{T}^d}\frac{|f(x)|}{\eta_f^5(x)}\sum_{i,j=1}^d \partial^2_{i,i} f(x) \partial_j f(x) \partial_j \left(\eta_f^2(x)\right) dx\\
& & - \frac{3}{2} \int_{\mathbb{T}^d}\frac{|f(x)|}{\eta_f^5(x)}\sum_{i,j=1}^d \partial^2_{i,j} f(x) \partial_i f(x) \partial_j\left(\eta_f^2(x)\right) dx\\
&=& \int_{\mathbb{T}^d} |f(x)| \left(\|\text{Hess}_x (f)\|^2-\text{Tr}\left(\text{Hess}_x(f)\right)^2\right)\frac{dx}{\eta_f(x)^3}\\
& &+\frac{3}{2} \int_{\mathbb{T}^d}\frac{|f(x)|}{\eta_f^5(x)} \left(\Delta f(x) \langle \nabla f(x), \nabla \eta_f^2(x)\rangle -\nabla f(x)^* \mathrm{Hess}_x f \nabla \eta_f^2(x)\right) dx.
\end{eqnarray*}
As a result, under the assumptions of $f\in C^3(\mathbb{T}^d,\mathbb R)$ we have established the desired formula
\[
\begin{array}{ll}
\mathcal H^{d-1}\left(\{f=0\}\right) & =\displaystyle{ - \frac{1}{2} \int_{\mathbb T^d}  \left( f(x) \Delta f(x)-  ||\nabla f(x)||^2 \right) \frac{|f(x)|}{\eta_f(x)^3} dx}\\
\\
&  \displaystyle{+\int_{\mathbb{T}^d} |f(x)| \left(\|\text{Hess}_x (f)\|^2-\text{Tr}\left(\text{Hess}_x(f)\right)^2\right)\frac{dx}{\eta_f(x)^3}}\\
\\
& \displaystyle{+\frac{3}{2} \int_{\mathbb{T}^d}\frac{|f(x)|}{\eta_f^5(x)}\sum_{i,j=1}^d \left(\partial^2_{i,i} f(x) \partial_j f(x)-\partial^2_{i,j} f(x) \partial_i f(x)\right)\partial_j\left(\eta_f^2(x)\right) dx.}
\end{array}
\]
Now, assume that $f \in C^2(\mathbb{T}^d,\mathbb R)$ is non-degenerated. One may approximate $f$ in the space $C^2(\mathbb{T}^d,\mathbb R)$ equipped with the norm $N(f)=\sup_{x\in\mathbb{T}^d} \left(|f(x)|+\|\nabla f(x)\|+\|\text{Hess}_x f\|\right)$ by a sequence $(g_n)_{n\ge 1}$ of functions belonging to $C^3(\mathbb{T}^d,\mathbb R)$. Since the right hand side of the the last equation does not have any term with a derivative of $f$ of order $3$, one may pass to the limit under the integral. Besides, since $f$ is non-degenerated, the nodal volume being a continuous functional for the weaker $C^1$ topology as established in \cite{angst2018universality}, one also has $\mathcal H^{d-1}\left(\{g_n=0\}\right)\to\mathcal H^{d-1}\left(\{f=0\}\right)$ as $n$ goes to infinity. Hence, we get that the last formula indeed holds for $f\in C^2(\mathbb{T}^d,\mathbb R)$. 
\end{proof}
 
 \newpage
\section{Almost sure results via Birkhoff ergodic Theorem}\label{birk-section}
 In this section we shall exploit the formula (\ref{crazy-formula}) in order to give almost sure statements regarding the number of roots on increasing domains. Let us first recall the Maruyama Theorem \cite[p. 76]{MR0448523}, which asserts that a stationary Gaussian process is ergodic if and only if its spectral measure has no atom. Associated with the well-known Birkhoff ergodic theorem, we have then the following theorem.
 
 \begin{thm}\label{Maruyama}
Let $(Y(x))_{x \in \R}$ be a $\R^d$-valued stationary Gaussian process whose spectral measure has no atom, and let $\Phi$ be a measurable function on $\R^d$ such that $\E\left(|\Phi(Y(0)|\right)<\infty$. Then one has,
 \begin{equation}\label{Ergo-Maruyama}
 \frac{1}{M}\int_{0}^{M}\Phi(Y(x))dx\xrightarrow[M\to\infty]{a.s.\,\text{and}\,\,L^1}~\E\left(\Phi(Y(0)\right).
 \end{equation}
 \end{thm}
 
 The next Theorem combines the fomula (\ref{crazy-formula}) with the previous statement, by considering $Y(x)=(f(x),f'(x),f''(x))$ and $\Phi(x,y,z)=\frac{x z -y^2}{x^2+y^2}.$ Since $f$ is a stationary Gaussian process, $f(x)$ is independent of $f'(x)$ and $\frac{1}{\sqrt{f^2(x)+f'(x)^2}}\in L^{p}(\mathbb{P})$ for every $p<2$. Since $|f(x)|<\sqrt{f^2(x)+f'(x)^2}$ and $f'(x)^2<f^2(x)+f'^2(x)$ we recover that
 
$$\frac{f'(x)^2-f(x)f''(x)}{f^2(x)+f'(x)^2}\in L^1(\mathbb{P}).$$

 \begin{thm}
Let $f$ be a real-valued stationary Gaussian process of class $\mathcal{C}^2$ and non-degenerated. We suppose that the spectral measure of $f$ has no atom. Then almost surely and in $\mathbb L^1$, as $T$ goes to infinity, we have
\[
\lim_{T\to +\infty} \frac{1}{T}\times \mathcal H^0\left(\{f=0\} \cap [0, T]\right) =\frac{\sqrt{\mathbb E[f'^2(0)]}}{\pi}.
\]
\end{thm}
\begin{proof}
Using Proposition 2, the number of zeros in $[0,T]$ can be expressed as 
\[
\mathcal H^0\left(\{f=0\} \cap [0, T]\right) =\frac{1}{\pi} \left[ \arctan \frac{f'(T)}{f(T)}-\arctan \frac{f'(0)}{f(0)} +   \int_0^T \frac{f'^2(x)-f(x)f''(x)}{f^2(x)+f'^2(x)}dx\right].
\]
Hence, the statement of the theorem is equivalent to the following almost sure and $\mathbb L^1-$convergence
\[
\lim_{T\to +\infty} \frac{1}{T} \int_0^T \frac{f'^2(x)-f(x)f''(x)}{f^2(x)+f'^2(x)}dx = \sqrt{\mathbb E[f'^2(0)]}.
\]
Denoting by $\mu$ the spectral measure associated with $f$, the (matrix-valued) spectral measure associated with the triplet $(f,f',f'')$ has a density with respect to $\mu$, hence no atoms.

\medskip

\medskip

Using Theorem \ref{Maruyama}, we get that, almost surely and in $\mathbb L^1$
\[
\lim_{T\to +\infty} \frac{1}{T} \int_0^T \frac{f'^2(x)-f(x)f''(x)}{f^2(x)+f'^2(x)}dx=  \mathbb E\left[\frac{f'^2(0)-f(0)f''(0)}{f^2(0)+f'^2(0)}\right].
\]
Now, remembering that the convergence takes place in $\mathbb L^1$, we have in fact
\[
\begin{array}{ll}
\displaystyle{\mathbb E\left[\frac{f'^2(0)-f(0)f''(0)}{f^2(0)+f'^2(0)}\right]} & \displaystyle{= \lim_{T \to +\infty}  \frac{\pi}{T} \times \mathbb E\left[ \mathcal H^0\left(\{f=0\} \cap [0, T]\right)\right]} \\
\\
& \displaystyle{=  \pi \times \mathbb E\left[ \mathcal H^0\left(\{f=0\} \cap [0, 1]\right) \right]}.
\end{array}
\]
By the classical Kac--Rice formula, we have then
\[
\mathbb E\left[ \mathcal H^0\left(\{f=0\} \cap [0, 1]\right) \right] = \frac{1}{\pi} \int_0^1 \sqrt{\frac{\mathbb E[f'^2(x)]}{\mathbb E[f^2(x)]}- \left(\frac{\mathbb E[f(x)f'(x)]}{\mathbb E[f^2(x)]}\right)^2}dx.
\]
By stationarity, the integrand is constant equal to $\sqrt{\mathbb E[f'^2(0)]}$, so that 
\[
\pi \times \mathbb E\left[ \mathcal H^0\left(\{f=0\} \cap [0, 1]\right) \right] = \sqrt{\mathbb E[f'^2(0)]},
\]
hence the result.
\end{proof}

We illustrate graphically this fact for the Gaussian process whose covariance function is given by the sinus cardinal.

\begin{figure}[ht]

\begin{center}

\includegraphics[scale=0.4]{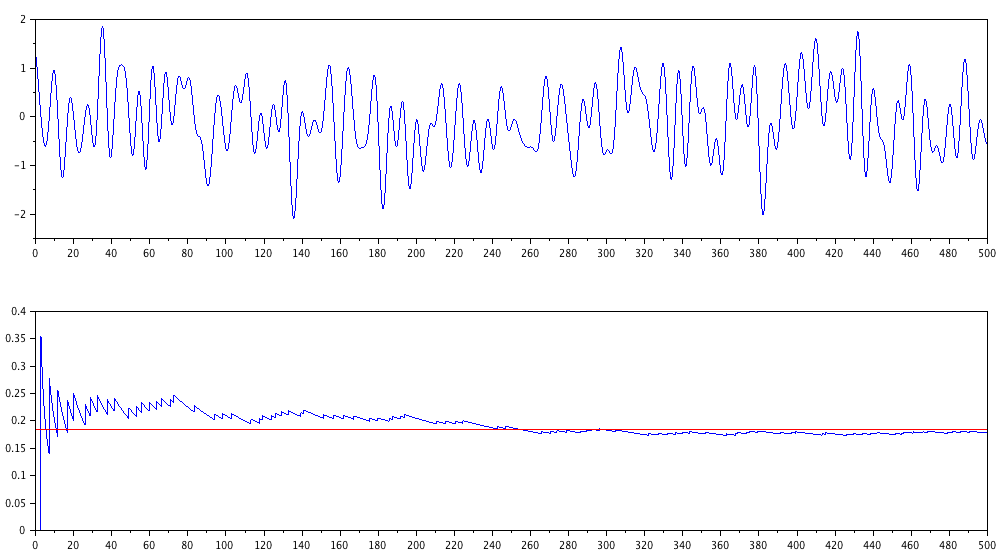}

\end{center}

\caption{A sample path of an approximation of the sine cardinal process $f$ obtained via the spectral representation method \cite{shinozuka91} and the corresponding renormalized zero counting function $t \mapsto t^{-1} \mathcal H^{0}\left( \{f =0\} \cap [0,t]\right)$. The red horizontal line corresponds to the theoretical limit $1/\pi \sqrt{3}\approx 0.184$.}

\end{figure}

\newpage 
\section{On the absolute continuity of the nodal volume}

In this last Section, we give the proof of our main absolute continuity results, namely Theorems \ref{Main theorem-periodic} and \ref{Mainpasperio} stated in the introduction. In order to keep the article as self contained as possible, we first recall in the next subsection the basics of Malliavin calculus and Bouleau--Hirsch criterion .

\subsection{A quick and self-contained introduction to Malliavin calculus}\label{Malliavin-introduction}
The purpose of this section is to introduce in a self-contained way the necessary material from Malliavin calculus theory that we will use in the sequel. All results presented here are classical, we nevertheless chose a particular \textit{gradient} called the sharp operator which has been previously introduced by N. Bouleau, see \cite[page 135]{bouleau2010dirichlet} or \cite[page 80]{bouleau2003error}. This choice of gradient is sometimes convenient since it preserves in some sense the \textit{Gaussian structure} of the objects with opposition with the standard choice of the literature to introduce an auxiliary Hilbert space of the form $L^2([0,T])$. Among others, in our case, we take benefit from this specific gradient in the proof of the proposition \ref{Integral is in the domain}, because the gradient of the underlying Gaussian field is itself continuous ensuring the convergence of Riemann sums.

\subsubsection{The sharp operator and its domain} Let us give $(X_i)_{i\ge 1}$ an i.i.d. sequence of standard Gaussian random variables and the underlying probability space $(\Omega,\mathcal{F},\mathbb{P})$. Without loss of generality, we shall assume that $\mathcal{F}=\sigma(X_i;i\ge 1)$. We will also need a copy $(\hat{\Omega},\hat{\mathcal{F}},\hat{\mathbb{P}})$ of this probability space as well as $(\hat{X}_i)_{i\ge 1}$ a corresponding i.i.d. sequence of standard Gaussian such that $\hat{\mathcal{F}}=\sigma(\hat{X}_i;i\ge 1)$.  For any $m\ge 1$ and any $F\in\mathcal{C}^1_{\text{Pol}}(\R^m,\R)$, the set of functions of $\mathcal{C}^1(\R^m,\R)$ whose gradient has a polynomial growth, one may define the following \textit{sharp operator}:
\begin{equation}\label{sharp operator}
\s F(X_1,\cdots,X_m)
:=\sum_{i=1}^m \partial_i F(X_1,\cdots,X_m) \hat{X_i}.
\end{equation}
Notice that the previous expression is in $\bigcap_{p\ge 1} L^p(\Omega,L^2(\hat{\mathbb{P}}))$ thanks to the fact that $\nabla F$ has a polynomial growth and that the Gaussian distribution has moments of any order. Let us introduce now,
\begin{equation}
\mathcal{P}:=\left\{ F(X_1,\cdots,X_m) \left|\right. F \in \mathcal{C}^1_{\text{Pol}}(\R^m,\mathbb{R}); m\ge 1\right\}.
\end{equation}
For any $F:=F(X_1,\cdots,X_m)\in\mathcal{P}$ we define the so-called \textit{Malliavin norm}
\begin{equation}\label{Malliavin-norm}
\|F\|_{1,p}:=\left(\E\left(|F|^p\right)+\E \left(\hat{\E}\left((\s F)^2\right)^{\frac{p}{2}}\right)\right)^{\frac{1}{p}},
\end{equation}
where $\E$ represents the expectation with respect to $\mathbb{P}$ and $\hat{\mathbb{E}}$ denotes the expectation for $\hat{\mathbb{P}}$. Again, the fact that the gradient of $F$ has a polynomial growth (hence so does $F$) implies that the quantity $\|F\|_{1,p}$ is well defined. We are now in position to extend the domain of the sharp operator to the set
\begin{equation}\label{DomainSharp}
\mathbb{D}_{1,p}:= \left\{ X \in L^p(\mathbb{P})\Big{|} \exists (F_n)_{n\ge 1}\in\mathcal{P}^{\mathbb{N}}, \,F_n\xrightarrow[n\to\infty]{L^p}~X\,\text{and}\,(F_n)_{n\ge 1}\,\text{is Cauchy for}\,\|\cdot\|_{1,p}\right\},
\end{equation}
which is the completion of the vector space $\mathcal{P}$ with respect to the norm $\|\cdot\|_{1,p}$. Now for any $X\in\mathbb{D}_{1,p}$ it remains to say what is $\s X$. Since $X\in\mathbb{D}_{1,p}$, one can find $(F_n)_{n\ge 1} \in \mathcal{P}^{\mathbb{N}}$ which converges toward $X$ and is a Cauchy sequence for the norm $\|\cdot\|_{1,p}$. In particular we get

$$\E \left(\hat{\E}\left((\s F_n-\s F_m)^2\right)^{\frac{p}{2}}\right)\xrightarrow[n,m\to\infty]~0.$$

Besides, in virtue of the definition \ref{sharp operator}, conditionally to the sigma-field $\mathcal{F}$, the random variables $\s F_n-\s F_m$ are Gaussian. Since all $L^p$ norms are equivalent for Gaussian random variables, the latter condition turns out to be equivalent to

$$\E \left(\hat{\E}\left(|\s F_n-\s F_m|^p\right)\right)\xrightarrow[n,m\to\infty]~0.$$

Hence, $\s F_n-\s F_m$ is a Cauchy sequence in $L^p(\mathbb{P}\otimes \hat{\mathbb{P}})$ which is a Banach space. The limit, denoted by $\s X \in L^p(\mathbb{P}\otimes \hat{\mathbb{P}})$, defines the sharp operator of the random variable $X$. Since $\E\left(\hat{\E}\left( |\s F_n- \s X|^p\right)\right)\to 0$, one can extract a subsequence such that $\mathbb{P}\text{-a.s.}$, $\hat{\E}\left( |\s F_n- \s X|^p\right)\to 0$. Since $\s F_n$ is Gaussian conditionally to the sigma field $\mathcal{F}$, the same property holds for $\s X$. Relying again on the equivalence of $L^p$ norms for Gaussian distribution, one also gets that $\s X \in L^p(\Omega, L^2 (\hat{\mathbb{P}}))$ as well as the two useful facts:
\begin{equation}\label{defi-gradient}
\begin{array}{ll}
&\s F_n \xrightarrow[n\to\infty]{L^p(\mathbb{P}\otimes \hat{\mathbb{P}})} \s X,\\
&\s F_n \xrightarrow[n\to\infty]{L^p(\Omega,L^2(\hat{\mathbb{P}}))} \s X.\\
\end{array}
\end{equation}

\subsubsection{Closability of the sharp operator} We must then show that this definition is unambiguous. Concretely, we must prove that $\s X$ does not depend on the chosen Cauchy sequence (for the norm $\|\cdot\|_{1,p}$) approximating $X$. This procedure consists in proving that the sharp operator is \textit{closable}, namely
$$
\left.
\begin{array}{ll}
&F_n \xrightarrow[n\to\infty]{L^p(\mathbb{P})}~0\\
&\exists Z \in L^p(\Omega,L^2(\hat{\mathbb{P}}))\,\text{s.t.}\,\E \left(\hat{\E}\left((\s F_n-Z)^2\right)^{\frac{p}{2}}\right)\to 0
\end{array}
\right\}
\Rightarrow
Z=0
$$

To do so we take $\theta: \mathbb{R}^m\to \R$ which belongs to the $\text{Schwartz class}$ and $\chi_M$ a smooth and compactly supported approximation of $\textbf{1}_{[-M,M]}$. We have the following integration by parts formula

\begin{equation}\label{IPP-closable}
\begin{array}{ccc}
\E\left(\theta(X_1,\cdots,X_m) \chi_M(X_i)\hat{\E} \left(\s F_n \hat{X_i}\right)\right)&=&\E\left(\theta(X_1,\cdots,X_m) \chi_M(X_i) X_i F_n \right)\\
&-&\E\left(F_n \hat{\E} \left( \s\left[\theta(X_1,\cdots,X_m) \chi_M(X_i)\right] \hat{X}_i\right)\right)
\end{array}
\end{equation}
By construction, $\chi_M(X_i) X_i$ is bounded, $\theta(X_1,\cdots,X_m)$ is also bounded hence we have

$$\E\left(\theta(X_1,\cdots,X_m) \chi_M(X_i) X_i F_n \right)\to 0.$$

Moreover, we have by the standard chain rule,

\begin{eqnarray*}
\s\left[\theta(X_1,\cdots,X_m) \chi_M(X_i)\right]&=&\s \left[\theta(X_1,\cdots,X_m)\right]\chi_M(X_i)\\
&+&\s \left[\chi_M(X_i)\right]\theta(X_1,\cdots,X_m)\\
&=&\sum_{j=1}^m \partial_j \theta(X_1,\cdots,X_m)  \hat{X}_j \chi_M(X_i)\\
&+&\theta(X_1,\cdots,X_m) \chi_M'(X_i) \hat{X}_i.
\end{eqnarray*}
Hence we get,
\begin{eqnarray*}
&&\E\left(F_n \hat{\E} \left( \s\left[\theta(X_1,\cdots,X_m) \chi_M(X_i)\right] \hat{X}_i\right)\right)\\
&=& \E\left(F_n \partial_i \theta(X_1,\cdots,X_m) \chi_M(X_i)\right)\\
&+& \E\left(F_n \theta (X_1,\cdots,X_m) \chi_M'(X_i) \right).
\end{eqnarray*}

Similarly, since $\|\theta\|_\infty+\|\nabla \theta \|_\infty +\|\chi_M\|_\infty+\|\chi_M'\|_\infty<\infty$ and $F_n\to 0$ in $L^p(\mathbb{P})$ we get

$$
\E\left(F_n \hat{\E} \left( \s\left[\theta(X_1,\cdots,X_m) \chi_M(X_i)\right] \hat{X}_i\right)\right)\to 0.
$$
We now treat the left hand side of the equation (\ref{IPP-closable}), we have

\begin{eqnarray*}
&&\left|\E\left(\theta(X_1,\cdots,X_m) \chi_M(X_i)\hat{\E} \left(\s F_n \hat{X_i}\right)\right)-\E\left(\theta(X_1,\cdots,X_m) \chi_M(X_i)\hat{\E} \left(Z \hat{X_i}\right)\right)\right|\\
&\le& \E\left(\left|\theta(X_1,\cdots,X_m) \chi_M(X_i)\right|\hat{\E} \left(\left|\s F_n-Z\right|\left| \hat{X_i}\right|\right)\right)\\
&\stackrel{\text{C.S.}}{\le}&\E\left(\left|\theta(X_1,\cdots,X_m) \chi_M(X_i)\right|\sqrt{\hat{\E} \left(\left(\s F_n-Z\right)^2\right)}\right)
\end{eqnarray*}

which tends to zero since $\theta(\cdot) \chi_M(\cdot)$ is bounded and since $\sqrt{\hat{\E} \left(\left(\s F_n-Z\right)^2\right)}$ tends to zero in $L^p$ by assumption and thus in $L^1$. Gathering all theses facts and passing to the limit in the equation (\ref{IPP-closable}) entails that

$$\E\left(\theta(X_1,\cdots,X_m) \chi_M(X_i)\hat{\E} \left(Z\hat{X_i}\right)\right)=0.$$

Letting first $M\to \infty$ and using dominated convergence (since $0\le \chi_M\le 1$) entails that for any $\theta \in \mathcal{S}(\R^m,\R)$ we have

$$\E\left(\theta(X_1,\cdots,X_m)\hat{\E} \left(Z\hat{X_i}\right)\right)=0.$$

Taking the conditional expectation with respect to $\mathcal{F}_m:=\sigma(X_1,\cdots,X_m)$ yields to

$$\E\left(\theta(X_1,\cdots,X_m)\E\left[\hat{\E} \left(Z\hat{X_i}\right)\left|\right.\mathcal{F}_m\right]\right)=0,$$

which by density arguments yields to 

$$\mathbb{P}\text{-a.s.},~~\E\left[\hat{\E} \left(Z\hat{X_i}\right)\left|\right.\mathcal{F}_m\right]=0.$$

Letting $m\to\infty$ and reminding that $\mathcal{F}=\sigma(X_i;i\ge 1)$ implies that

$$\mathbb{P}\text{-a.s.},~~\hat{\E} \left(Z\hat{X_i}\right)=0.$$

The latter being valid for any $i\ge 1$ we deduce that

$$\forall n\ge 1,\,\mathbb{P}\text{-a.s.},~~\hat{\E} \left(Z ~\s F_n\right)=0.$$

But $\s F_n $ tends to $Z$ in $L^p(\Omega,L^2(\hat{\mathbb{P}}))$, hence up to extracting a subsequence we may assume that $\mathbb{P}\text{-a.s.}$, $\hat{\E}\left( (Z-\s F_n)^2\right)\to 0$. On the other hand, using Cauchy Schwarz inequality we also have

$$\left|\hat{\E}\left(Z^2-Z~\s F_n \right)\right|\le \sqrt{\hat{\E}\left(Z^2\right)}\sqrt{\hat{\E}\left((Z-\s F_n)^2\right)} \to 0.$$
Finally, we obtain that $\mathbb{P}\text{-a.s.}$ we have $\hat{\E}(Z^2)=0$ and finally the desired conclusion that $\mathbb{P}\otimes\hat{\mathbb{P}}\text{-a.s.},\,Z=0.$

\begin{rem}
Based on the definition \ref{sharp operator}, one has the following formal interpretation of $\s$. Take $\Phi:=\Phi(X_1,\cdots,X_i,\cdots)$ some functional of the sequence $(X_i)_{i\ge 1}$ which belongs to $\mathbb{D}^{1,p}$ for some $p\ge 1$. Then one has formally:
\begin{equation}\label{Meaningofthederivative}
\s \Phi=\lim_{\epsilon\to 0} \frac{\Phi\left(X_1+\epsilon \hat{X}_1,\cdots,X_i+\epsilon \hat{X}_i,\cdots\right)-\Phi(X_1,\cdots,X_i,\cdots)}{\epsilon}.
\end{equation}
In some sense, the sharp operator represents the directional derivative along an independent copy of the input.
\end{rem}

\paragraph{Main properties of the sharp operator:} In this paragraph we gather the main properties of the sharp operator that we will use in our proof of the existence of densities for the nodal volumes. The following property may be seen as an analogue of the usual chain rule for the sharp operator $\s$.

\begin{prop}\label{Chain-Rule-Proposition}
Let $(X_1,\cdots,X_m)$ in the domain $\mathbb{D}^{1,p}$ for some $p\ge 1$. Let $\Psi\in\mathcal{C}_b^1(\R^m,\R)$, that is to say $\Psi$ is continuously differentiable with a bounded gradient. Then, $\Psi(X_1,\cdots,X_m)\in\mathbb{D}^{1,p}$ and we have
\begin{equation}\label{Chain-Rule}
\s \Psi(X_1,\cdots,X_m)=\sum_{i=1}^m \partial_i \Psi(X_1,\cdots,X_m) \s X_i.
\end{equation}
\end{prop}

\begin{proof}
Coming back to our definition of the gradient and the properties (\ref{defi-gradient}), for each $i\in\{1,\cdots,m\}$, one can find a sequence $(F_{n,i})_{n\ge 1}$ in the space $\mathcal{P}$ which converges towards $X_i$ for the norm $\|\cdot\|_{1,p}$. Moreover, since $\Psi\in\mathcal{C}_b^1(\R^m,\R)$, we also get $\Psi(F_{n,1},\cdots,F_{n,m})\in\mathcal{P}$ and the usual chain rule of differential calculus asserts that
\begin{equation*}
\s \Psi(F_{n,1},\cdots,F_{n,m})=\sum_{i=1}^m \partial_i \Psi(F_{n,1},\cdots,F_{n,m}) \s F_{n,i}.
\end{equation*}
Up to extracting a subsequence, we may assume that $F_{n,i}\to X_i$ almost surely for every $i\in\{1,\cdots,m\}$. Besides, $\s F_{n,i}$ converges towards $\s X_i$ in $L^p(\Omega,L^2(\mathbb{\hat{P}}))$ (as well as in $L^p(\mathbb{P}\otimes \hat{\mathbb{P}})$ since these norms are equivalent in our framework) and $\nabla \Psi$ is bounded and continuous. This ensures that
$$\sum_{i=1}^m \partial_i \Psi(F_{n,1},\cdots,F_{n,m}) \s F_{n,i} \xrightarrow[n\to\infty]{L^p(\Omega,L^2(\mathbb{\hat{P}}))}~\sum_{i=1}^m \partial_i \Psi(X_1,\cdots,X_m) \s X_i.
$$
Relying on the global Lipschitz property of $\Psi$ we also have $$\Psi(F_{n,1},\cdots,F_{n,m})\xrightarrow[n\to\infty]{L^p(\mathbb{P})}~\Psi(X_1,\cdots,X_m),$$ which achieves the proof.
\end{proof}

The cornerstone of our approach is the following criterion of density which is due to N. Bouleau. Originating from Dirichlet forms theory, it is customarily called the property of \textit{density of the energy image}.

\begin{thm}\label{EID}
For any $p\ge 1$ and any $X\in \mathbb{D}^{1,p}$ we have
\begin{equation}\label{EIDeq}
X_*\left(\sqrt{\hat{\E}\left((\s X)^2\right)}d\mathbb{P}\right)\ll\lambda.
\end{equation}
where $\mu \ll \nu$ means that $\mu$ is absolutely continuous with respect to $\nu$ and $\lambda$ stands for the Lebesgue measure on $\mathbb R$.
\end{thm}
\begin{proof}
We reproduce here, using our notations, the original proof of Nicolas Bouleau which has been taken from \cite[page 42]{bouleau2003error} and which we slightly adapt. Take $A$ a Borel set with zero Lebesgue measure and denote by $\mu$ the probability $X_*\left(\sqrt{\hat{\E}\left((\s X)^2\right)}d\mathbb{P}\right)$. The key idea is to consider a sequence $\phi_n$ of continuous functions, such that $0\le \phi_n\le 1$ and which converge towards $\textbf{1}_A$ almost surely for the mixed positive measure $\lambda+\mu$. The existence of such a sequence of functions is ensured by the density of continuous compactly supported functions in the space $L^1(\lambda+\mu)$. Now, one sets $\Phi_n(x)=\int_0^x \phi_n(t) dt$. We shall prove that $\Phi_n(X)$ is a Cauchy sequence in the space $\mathbb{D}^{1,p}$. First, we notice that $\Phi_n$ is one-Lipschitz and vanishes at zero, hence $|\Phi_n(X)|\le |X|$. Besides, using dominated convergence we have
$$\Phi_n(X)=\int_0^X \phi_n(t) dt \xrightarrow[n\to\infty]~\int_0^X \textbf{1}_A(t) dt=0.$$
Gathering these two facts ensure that $\E\left(|\Phi_n(X)|^p\right)\to 0.$ On the other hand, one has by the usual chain rule that a.s.
\begin{eqnarray*}
\sqrt{\hat{\E}\left(\left(\s\left(\Phi_n(X)-\Phi_m(X)\right)\right)^2\right)}&=&\left|\phi_n(X)-\phi_m(X)\right| \sqrt{\hat{\E}\left(\left(\s X\right)^2\right)}.\\
&&\xrightarrow[n,m\to\infty]~\left|\textbf{1}_A(X)-\textbf{1}_A(X)\right| \sqrt{\hat{\E}\left(\left(\s X\right)^2\right)}\\
&=&0.
\end{eqnarray*}
Recall that $\phi_n$ is uniformly bounded and that $\sqrt{\hat{\E}\left(\left(\s X\right)^2\right)}$ is in $L^p(\mathbb{P})$, one can use again dominated convergence. It follows that

$$\E\left(\sqrt{\hat{\E}\left(\left(\s\left(\Phi_n(X)-\Phi_m(X)\right)\right)^2\right)}^p\right)\xrightarrow[n,m\to\infty]~0.$$
As a result, $\Phi_n(X)$ is a Cauchy sequence in the Banach space $\mathbb{D}^{1,p}$ and the limit has to be zero, since it tends to zero in $L^p$. It remains to says that
\begin{eqnarray*}
\E\left(\sqrt{\hat{\E}\left(\left(\s\Phi_n(X)\right)^2\right)}^p\right)\xrightarrow[n\to\infty]~\E\left(\textbf{1}_A(X)\sqrt{\hat{\E}\left(\left(\s X\right)^2\right)}\right),
\end{eqnarray*}
which is necessarily zero since $\Phi(X)$ tends to zero in $\mathbb{D}^{1,p}$.
\end{proof}
We end our introduction by the following chain rule criterion for \textit{Lipschitz} regularity. It is important in our framework since it will allow us to take the Malliavin derivative of the absolute value of a Gaussian process.

\begin{prop}\label{Chain-rule-Lipschitz}Let us assume here that $p>1$. Take $X\in\mathbb{D}^{1,p}$ and $\Psi$ a globally Lispchitz function, then $\Psi(X)\in\mathbb{D}^{1,p}$ and we have
\begin{equation}\label{CRLip-eq}
\s \Phi(X)=\Phi'(X) \s X,
\end{equation}
where $\Phi'$ is any Borelian representation of the derivative of $\Phi$, which exists Lebesgue almost everywhere according to the Rademacher Theorem.
\end{prop}
\begin{proof}
Take $p_\epsilon=\frac{1}{\epsilon} p(\frac{\cdot}{\epsilon})$ a regularizing kernel and set $\Phi_\epsilon=\Phi \star p_\epsilon$, we have $\Phi_\epsilon'\xrightarrow[\epsilon\to 0]{L^{1,loc}(\lambda)} \Phi'$. Hence, we may find $\epsilon_n\to 0$ such that $\Phi_{\epsilon_n}'$ converges Lebesgue almost surely towards $\Phi'$. On the other hand, $\Phi_{\epsilon_n}'(X) \s X$ is a bounded sequence in $L^p(\mathbb{P}\otimes \hat{\mathbb{P}})$ and one can extract a subsequence which converges weakly, here the fact that $p>1$ plays a role since the unit ball of $L^\infty$ is not sequentially compact. Using Mazur Lemma one can even find convex linear combinations of terms of the form $\Phi'_{\epsilon_n}(X)\s X$ which converge \textit{strongly} in $L^p(\mathbb{P}\otimes\hat{\mathbb{P}})$ towards a term of the form $g(X) \s X$, for some Borelian function $g$ to be determined later. Now, relying on proposition \ref{Chain-Rule-Proposition}, we may write

$$\s \Phi_{\epsilon_n}(X)=\Phi'_{\epsilon_n}(X) \s X.$$
We then have:
{\small
\begin{eqnarray*}
&&\E\hat{\E}\left(\left|\sum_{i=p_n}^{q_n}\alpha_{i,n}\left(\Phi'_{\epsilon_i}(X)-g(X)\right)\s X\right|^p\right)\sim \E\left(\hat{\E}\left(\left(\sum_{i=p_n}^{q_n}\alpha_{i,n}\left(\Phi'_{\epsilon_i}(X)-g(X)\right)\s X\right)^2\right)^{\frac{p}{2}}\right)\\
&&=\E\left(\left|\sum_{i=p_n}^{q_n}\alpha_{i,n}\left(\Phi'_{\epsilon_i}(X)-g(X)\right)\right|^p \hat{\E}(\s X^2)^{\frac p 2}\right) \xrightarrow[n\to\infty]~ 0.
\end{eqnarray*}
According to the Theorem \ref{EID}, if $\hat{\E}\left(\s X ^2\right) \neq 0$ then $X$ belongs to the support of the one dimensional Lebesgue measure and from the previous discussion we can assert that 
$$
\E\left(\left|\sum_{i=p_n}^{q_n}\alpha_{i,n}\left(\Phi'_{\epsilon_i}(X)-g(X)\right)\right|^p \hat{\E}(\s X^2)^{\frac p 2}\right)\to
\E\left(\left|\Phi'(X)-g(X)\right|^p \hat{\E}(\s X^2)^{\frac p 2}\right)=0.
$$
}
It is enough to say that $\sum_{i=p_n}^{q_n}\alpha_{i,n}\Phi_{\epsilon_i}(X)$ is a Cauchy sequence in $\mathbb{D}^{1,p}$ which tends to $\Phi(X)$. The previous reasoning ensures us that $\s \Phi(X)= \Phi'(X) \s X$. We notice that the result does not depend on the specific choice of a Borelian representation of $\Phi'$, which also appears in the previous proof.
\end{proof}
\subsection{Computing the Malliavin derivative of the nodal volume}

Let us now describe how the closed Kac--Rice formulas established in Section \ref{sec.Kac} associated with the Bouleau--Hirsch criterion allow to derive the absolution continuity of the nodal volume. 

\subsubsection{Generating Gaussian processes on suitable probability spaces} On a probability space $(\Omega,\mathcal{G},\mathbb{P})$, we consider a continuous Gaussian process $(X_t)_{t\in\mathbb{T}^d}$. The very first step of our approach is classic and consists in generating $(X_t)_{t\in\mathbb{T}^d}$ in a probability space $(\Omega,\mathcal{F},\mathbb{P})$ where $\mathcal{F}=\sigma(Y_i;i\ge 1)$ and where $(Y_i)_{i\ge 1}$ is an i.i.d. sequence of standard Gaussian random variables. Relying on the continuity, $(X_t)_{t\in\mathbb{T}^d}$ is also measurable with respect to the sigma-field $\mathcal{F}:=\sigma \left( X_t; t\in\mathbb{Q}^d\right)$. Moreover, using Gram-Schmidt procedure, one can find an i.i.d. sequence of standard Gaussian $(Y_i)_{i\ge 1}$ such that $\text{Adh}_{L^2}\left(\text{Vect}(Y_i;i\ge 1)\right)=\text{Adh}_{L^2}\left(\text{Vect}(X_t;t\in\mathbb{Q}^d)\right)$ and necessarily $\sigma(Y_i;i\ge 1)=\mathcal{F}$. Hence the claim follows. Now, in order to use the tools of Malliavin calculus introduced in the subsection \ref{Malliavin-introduction}, we give us a copy of the previous probability space, namely $(\hat{\Omega},\hat{\mathcal{F}},\hat{\mathbb{P}})$ with $\hat{\mathcal{F}}=\sigma(\hat{Y}_i;i\ge 1)$. For any $t\in\mathbb{T}^d$ there exists $(\alpha_i(t))_{i\ge 1} \in l^2(\mathbb{N}^*)$ such that
\[
X_t\stackrel{L^2}{=} \sum_{i\ge 1} \alpha_i (t) Y_i, \quad \text{and similarly} \quad \hat{X}_t\stackrel{L^2}{=} \sum_{i\ge 1} \alpha_i (t) \hat{Y}_i.
\]
It follows that $(X_t)_{\mathbb{T}^d}\stackrel{\text{Law}}{=} (\hat{X}_t)_{t\in\mathbb{T}^d}$. Besides, up to a modification, we shall always assume that the process $(\hat{X}_t)_{t\in\mathbb{T}^d}$ has the same regularity as $(X_t)_{t\in\mathbb{T}^d}$. Without loss of generality, we shall assume that the various Gaussian processes considered below are generated on a suitable probability space and that we can deploy the Malliavin calculus tools introduced in section \ref{Malliavin-introduction}.

\subsubsection{Some technical lemmas}The next lemmas are essentially technical and will be used in the computations of the Malliavin derivatives of the nodal volume. For the sake of clarity we introduce the two followings domains:

$$
\begin{array}{ll}
&\displaystyle{\mathbb{D}^{1,\infty}=\bigcap_{p\ge 1} \mathbb{D}^{1,p}},\\
&\displaystyle{\mathbb{D}^{1,p^{-}}=\bigcap_{1\le q<p}\mathbb{D}^{1,q}}.
\end{array}
$$

\begin{lemma}\label{CR-technical}
Let $p>1$, let $\alpha>0$ and $(X_1,\cdots,X_m,X_{m+1})\in\mathbb{D}^{1,\infty}$. Let us assume that
$$
\begin{array}{ll}
(i) & \mathbb{P}\text{-a.s.},\, X_{m+1}>0,\\

&\\

(ii) & \frac{X_1 X_2 \cdots X_m}{X_{m+1}^\alpha}\in L^p(\mathbb{P}),\\

&\\

(iii) & \displaystyle{\sum_{i=1}^m |\s X_i|\left|\frac{ \prod_{j\neq i} X_j}{X_{m+1}^\alpha}\right|+ \left(\prod_{i=1}^m|X_i| \right) \left|\frac{\s X_{m+1}}{X_{m+1}^{\alpha+1}}\right| \in L^p(\mathbb{P}\otimes \hat{\mathbb{P}})}.\\
\end{array}
$$

Then we get that $\frac{X_1 X_2 \cdots X_m}{X_{m+1}}\in\mathbb{D}^{1,p}$ as well as the formula
\begin{equation}\label{Technical1}
\s \left[\frac{X_1 X_2 \cdots X_m}{X_{m+1}^{\alpha+1}}\right]=\sum_{i=1}^m \s X_i \frac{\prod_{j\neq i} X_j}{X_{m+1}^{\alpha}}-\alpha\frac{\s X_{m+1}}{X_{m+1}^{\alpha+1}}\prod_{i=1}^m X_i.
\end{equation}
\end{lemma}

\begin{proof}
Set $\theta_\epsilon(x):=\frac{1}{\left(x^2+\epsilon\right)^{\frac{\alpha}{2}}}$ and let $\chi_M$ be a smooth function satisfying 
\[
\begin{array}{ll}
\chi_M(x)=x \;\; \text{on}\;\; [-M,M], & \chi_M=2M \;\; \text{on}\;\;  [2M,+\infty[, \\
\\
\chi_M=-2M \;\; \text{on}\;\;]-\infty,-2M], & |\chi_M(x)|\le |x|, \;\; \forall x \in \mathbb R.
\end{array}
\]
 One can use the content of proposition \ref{Chain-Rule-Proposition} and get that $\theta_\epsilon (X_{m+1})\prod_{i=1}^m \chi_M(X_i)\in \mathbb{D}^{1,p}$ together with the formula:

\begin{eqnarray*}
\s\left[\prod_{i=1}^m \chi_M(X_i)\theta_{\epsilon}(X_{m+1})\right]&=&\sum_{i=1}^m \chi_M'(X_i) \s X_i \left(\prod_{j\neq i} \chi_{M}(X_j)\right)\theta_{\epsilon}(X_{m+1})\\
&+&\left(\prod_{i=1}^m \chi_M(X_i)\right) \theta_\epsilon'(X_{m+1}) \s X_{m+1}.
\end{eqnarray*}

Let us notice that
\begin{itemize}
\item $\sup_{x\in\R,M>0}\|\chi_M'(x)\|_\infty<+\infty$, 
\item $\forall x\in \R, |\chi_M(x)|\le |x|$,
\item $\forall x>0, |\theta_\epsilon(x)|\le \frac{1}{x^\alpha},$
\item $\forall x>0, |\theta_\epsilon'(x)|\le \frac{\alpha}{x^{\alpha+1}}$.
\end{itemize}

Then one may use the dominated convergence Theorem (with domination given by the above estimates and our assumptions (ii)-(iii)) to ensure that

\begin{eqnarray*}
&&\prod_{i=1}^m \chi_M(X_i)\theta_{\epsilon}(X_{m+1})\xrightarrow[M\to\infty,\epsilon\to 0]{L^p(\mathbb{P})}~\frac{\prod_{i=1}^m X_i}{X_{m+1}^\alpha}\\
&&\s\left[\prod_{i=1}^m \chi_M(X_i)\theta_{\epsilon}(X_{m+1})\right]\xrightarrow[M\to\infty,\epsilon\to 0]{L^p(\mathbb{P\otimes \hat{\mathbb{P}}})}~\sum_{i=1}^m \s X_i \frac{\prod_{j\neq i} X_j}{X_{m+1}^{\alpha}}-\alpha\frac{\s X_{m+1}}{X_{m+1}^{\alpha+1}}\prod_{i=1}^m X_i.
\end{eqnarray*}

The result follows from the completeness of $\mathbb{D}^{1,p}$ with respect to the norm $\|\cdot\|_{1,p}$.
\end{proof}

\begin{lemma}\label{Integrand is in the domain}
We will assume here that $f\in \mathcal{C}^2(\mathbb{T}^d,\R)$ is a stationary Gaussian process  and we make here the crucial assumption that the covariance matrix of the Gaussian vector $(f,\nabla f)$ is non-degenerated. Then, for any $d\ge 3$ and any $x\in\mathbb{T}^d$ we have

\begin{eqnarray*}
&\text{(a)}&\left( f(x) \Delta f(x)-  ||\nabla f(x)||^2 \right) \frac{|f(x)|}{\eta_f(x)^3}\in \mathbb{D}^{1,\frac{d+1}{3}-}\\
&\text{(b)}&|f(x)| \left(\|\text{Hess}_x (f)\|^2-\text{Tr}\left(\text{Hess}_x(f)\right)^2\right)\frac{1}{\eta_f(x)^3} \in \mathbb{D}^{1,\frac{d+1}{3}-}\\
&\text{(c)}&\frac{|f(x)|}{\eta_f^5(x)}\sum_{i,j=1}^d \left(\partial^2_{i,i} f(x) \partial_j f(x)-\partial^2_{i,j} f(x) \partial_i f(x)\right)\partial_j\left(\eta_f^2(x)\right)\in\mathbb{D}^{1,\frac{d+1}{3}-}
\end{eqnarray*}
\end{lemma}
\begin{proof}
Let us fix $x\in\mathbb{T}^d$, the stationarity ensures us that the forthcoming computations do not depend on $x$.
Recall that $\eta_f^2(x)=f^2(x)+\|\nabla f(x)\|^2$ is almost surely positive in virtue of \cite[p. 132 prop. 6.12]{azais2009level}, and denote by $\Sigma$ the covariance matrix of the $(d+1)-$dimensional Gaussian vector $(f(x),\nabla f(x))$ which is assumed to be invertible. Set $\mu$ the minimum eigenvalue of $\Sigma$, one has

\begin{eqnarray*}
\E\left(\frac{1}{\eta_f^{3\beta}}\right)&=&\int_{\R^{d+1}} \frac{1}{\|\Vec{x}\|^{3\beta}}\exp\left(-\frac{1}{2}~ {}^t \Vec{x}~\Sigma ~\Vec{x} \right)\frac{dx}{\sqrt{(2\pi)^{d+1}\det\Sigma}}\\
&\le& \int_{\R^{d+1}}\frac{1}{\|\Vec{x}\|^{3\beta}}\exp\left(-\frac{1}{2}\mu\|\Vec{x}\|^2\right)\frac{dx}{\sqrt{(2\pi\mu)^{d+1}}}\\
&=&\frac{1}{\sqrt{(2\pi\mu)^{d+1}}}\int_{\mathbb{S}^d}\int_0^\infty \frac{1}{r^{3\beta}}e^{-\frac{\mu}{2}r^2} r^d dr d\sigma(u).
\end{eqnarray*}
This quantity is finite if and only if $d>3\beta-1$ or else $\beta<\frac{d+1}{3}$. As a result, for $1\le\beta<\frac{d+1}{3}$ on has $\eta_f^{-3}\in L^\beta.$ In order to treat the cases $(a)-(b)-(c)$, one must notice that $|f|\le \eta_f$ and that for all $i\in\{1,\cdots,d\}$ we have also $|\partial_i f|\le \eta_f$. Next, for each term appearing on the expressions involved in cases $(a)-(b)-(c)$, one must count $n_1$ the number of factors in the set $\{f,\nabla f\}$ appearing at the numerator and the power $n_2$ of $\eta_f$ at the denominator. The worst case possible is given by $n_2-n_1=2$ which requires the integrability of $\eta_f^{-3}$ after derivation (and $\eta_f^{-2}$ before) and gives the threshold $\frac{d+1}{3}$, according to the previous computation. Hence all the integrability conditions will be satisfied and one can applies the Lemma \ref{CR-technical}. As a matter of fact, all terms $(a)$, $(b)$ and $(c)$ actually belong to $\mathbb{D}^{1,\frac{d+1}{3}-}$.

\end{proof}

\begin{rem}\label{moment}
The previous proof shows that the Malliavin derivative of $\mathcal{H}^{d-1}(f=0)$ has moment of any order less than $\frac{d+1}{3}$ which gives the announced domain $\mathbb{D}^{1,\frac{d+1}{3}}$. If one is only interested in the integrability of $\mathcal{H}^{d-1}(f=0)$, the latter holds as soon as $\eta_f^{-2}$ is integrable. Hence, we have 
$$\mathcal{H}^{d-1}(f=0)\in  L^{\frac{d+1}{2}-}(\mathbb{P}).$$
As noticed in the introduction, the integrability of the nodal volume thus increases with the dimension $d$, under the sole regularity $\mathcal C^2$. 
\end{rem}

\begin{lemma}\label{Integral is in the domain}
Under the assumptions of the previous lemma, for any $d\ge 3$ we have $\mathcal{H}_{d-1}\left(f=0\right)\in \mathbb{D}^{1,\frac{d+1}{3}-}$ and we have

\begin{equation}\label{sharpduvolume}
\begin{array}{ll}
&\s\mathcal H^{d-1}\left(\{f=0\}\right)  =\displaystyle{ - \frac{1}{2} \int_{\mathbb T^d} \s \left[\left( f(x) \Delta f(x)-  ||\nabla f(x)||^2 \right) \frac{|f(x)|}{\eta_f(x)^3} \right]dx}\\
\\
&  \displaystyle{+\int_{\mathbb{T}^d} \s\left[|f(x)| \left(\|\text{Hess}_x (f)\|^2-\text{Tr}\left(\text{Hess}_x(f)\right)^2\right)\frac{1}{\eta_f(x)^3}\right] dx}\\
\\
& \displaystyle{+\frac{3}{2} \int_{\mathbb{T}^d}\s\left[\frac{|f(x)|}{\eta_f^5(x)}\sum_{i,j=1}^d \left(\partial^2_{i,i} f(x) \partial_j f(x)-\partial^2_{i,j} f(x) \partial_i f(x)\right)\partial_j\left(\eta_f^2(x)\right)\right] dx.}
\end{array}
\end{equation}

\begin{proof}
The proof consists in approximating the integrals by Riemann sums. Indeed, by assumption, expressions $(a)-(b)-(c)$ are continuous almost surely, hence their Riemann sums converge almost surely. Besides, all the expressions $(a)-(b)-(c)$ are bounded in $L^\beta$ for every $\beta<\frac{d+1}{3}$, hence so do the Riemann sums. Theses two facts ensure that the Riemann sums converge in $L^\beta$ for every $\beta<\frac{d+1}{3}$. On the other hand, if one applies the sharp operator to theses Riemann sums, one obtains in turn the Riemann sum of the sharp operator applied to the expressions $(a)-(b)-(c)$ which are also both continuous in $x$ and bounded in $L^\beta$ for every $\beta<\frac{d+1}{3}$. Hence the same reasoning holds. The completeness of the domains $\mathbb{D}^{1,p}$ endowed with their natural norm $\|\cdot\|_{1,p}$ gives the desired claim.
\end{proof}
\end{lemma}

\subsection{On the non-degeneracy of the Malliavin derivative}

Let us begin with the following well-known result of Malliavin calculus. Again, in order to remain self-contained we sketch a proof.

\begin{lemma}
Let $X\in\mathbb{D}^{1,p}$ for $p>1$ and we assume that $\mathbb{P}\otimes\hat{\mathbb{P}}\text{-a.s.}$ we have $\s X=0.$ Then, the random variable $X$ is constant.
\end{lemma}
\begin{proof} Set $M>0$ and $\chi_M$ a $\mathcal{C}^\infty$ function satisfying $\chi_M(x)=x$ on $[-M,M]$ and $\chi_{M}(x)=M+1$ on $[M+1,\infty[$ and $-M-1$ on $]-\infty,-M-1]$. One has $\chi_M(X)\in\mathbb{D}^{1,2}$ and the Poincaré inequality coupled with the usual chain rule ensures that
\begin{equation*}
\text{Var}(\chi_M(X))\le \E\left(\chi_M'(X)^2\hat{\E}\left(\s X^2\right)\right)=0.
\end{equation*}
As a result, $\chi_M(X)$ is a constant random variable which tends almost surely to $X$. Hence, $X$ is constant as it was expected.
\end{proof}

Besides, based on Lemma \ref{Integral is in the domain}, one knows that the nodal volume belongs to some Malliavin spaces $\mathbb{D}^{1,p}$. Thus, if one can proves that the nodal volume is not a constant random variable, we can derive that its Malliavin derivative is not almost surely zero. Relying on Theorem \ref{EID}, it implies  that the distribution of $X:=\mathcal{H}(f=0)$ is not singular with respect to the Lebesgue measure. Indeed, assume that the distribution of $X$ is singular, one can find a Borel set $A$, negligible  with respect to the Lebesgue measure, such that $\mathbb{P}\left(X\in A\right)=1$. The statement (\ref{EIDeq}) ensures that
$$\E\left( \textbf{1}_A(X) \hat{\E}\left(\s X^2\right)\right)=0,$$
which in turn gives that
$$\E\hat{\E}\left(\s X^2\right)=0.$$
Hence the Malliavin derivative would be almost surely zero and the nodal volume constant.

\begin{rem}\label{Notconstantrem}
It is in general easy to prove that the nodal volume is not a constant random variable. Let us illustrate this in fact in the framework of random trigonometric sums. Suppose that $\Lambda$ is a finite (non-empty) set in $\mathbb Z^d$ and consider the polynomial
\[
f(x)=\sum_{(k_1,\cdots,k_d)\in \Lambda} c_{k_1,\cdots,k_d} G_{k_1,\cdots,k_d} \cos(k_1 x_1)\cdots \cos(k_d x_d),
\] 
where $c_{k_1,\cdots,k_d}$ are non-zero deterministic coefficients and $\left(G_{k_1,\cdots,k_d}\right)_{k_1,\cdots,k_d}$ is an i.i.d. sequence of Gaussian random variables. For some fixed multi-index $(l_1,\cdots,l_d) \in \Lambda$ and some $\epsilon>0$ the following event has positive probability
\[
A_{l_1,\cdots,l_d}^\epsilon:=\left\{ \left\|\sum_{(k_1,\cdots,k_d)\neq (l_1,\cdots,l_d)} c_{k_1,\cdots,k_d} G_{k_1,\cdots,k_d} \cos(k_1 x_1)\cdots \cos(k_d x_d)\right\|_{\mathcal{C}^1}<\epsilon\right\},
\]
since it may happen that all the Gaussian coefficients of the previous sum are small. On the other hand, if the nodal volume of $f$ is constant, than it must equal its conditional expectation with respect to $A_{l_1,\cdots,l_d}^\epsilon$. Let us recall that the nodal volume is continuous with respect to the $\mathcal{C}^1$ topology which has been proved in \cite{angst2018universality}. It follows that 
$$\mathcal{H}^{d-1}(f=0)=\E\left[ \mathcal{H}^{d-1}(f=0)~~ \left|\right. ~~A_\epsilon\right]\xrightarrow[\epsilon\to 0]~\mathcal{H}^{d-1}\left(\cos(l_1 x_1)\cos(l_2 x_2)\cdots \cos(l_d x_d)=0\right).$$
Choosing another multi-index $(l_1,\cdots,l_d) \in \Lambda$ with a different deterministic nodal volume brings a contradiction.
\end{rem}

\subsection{Understanding the singularity}

In this subsection, we will seek for conditions implying that $\mathcal H^{d-1}\left(\{f=0\}\right)$ has a density. One very likely condition is that the nodal volume is absolutely continuous conditionally to the fact that $f$ vanishes. Indeed, as discussed in Remark \ref{diracenzero}, we recall that the nodal volume distribution has an atom on zero if the underlying Gaussian field can be strictly positive on $\mathbb{T}^d$. Our strategy consists in proving that $f$ has a constant sign if the Malliavin derivative is zero. The computations reveal a second order differential operator whose kernel contains both $1$ and $\text{sign}(f)$. Therefore, if this kernel is of dimension one, one deduces the desired non-degeneracy.

We give us $1\le p<\frac{d+1}{3}$ which exists as soon as $d\ge 3$. The forthcoming computations will hold in the domain $\mathbb{D}^{1,p}$. Relying on the property of the energy image density (\ref{EIDeq}), one is left to show that

\begin{equation}{\label{Nondege}}
\mathbb{P}\text{-a.s.},\,\,\hat{\E}\left(\left[\s\mathcal H^{d-1}\left(\{f=0\}\right)\right]^2\right)\neq 0.
\end{equation}

To do so, we fix some $\omega\in\Omega$ such that

$$\hat{\E}\left(\left[\s\mathcal H^{d-1}\left(\{f=0\}\right)\right]^2\right)=0.$$

\noindent Combining the equation (\ref{Technical1}) with Lemma \ref{Integral is in the domain}, one may compute in an explicity way the Malliavin derivative of the nodal volume which is given by $\s\mathcal H^{d-1}\left(\{f=0\}\right)$.  Indeed, using the chain rule property of the sharp operator, one may find some explicit processes $\Psi_0$, $(\Psi_i)_{1\le i \le d},(\Psi_{i,j})_{1\le i,j\le d}$, which are all of the form

\begin{equation}\label{form-auxiliary-functions}
\Psi_0:=\text{sign}(f)\,\frac{P_0}{\eta_f^{7}},\Psi_i :=\text{sign}(f)\,\frac{P_i}{\eta_f^7},\,\Psi_{i,j}:=\text{sign}(f)\,\frac{P_{i,j}}{\eta_f^7}
\end{equation}
with $(P_0,(P_i)_{1\le i \le d},(P_{i,j})_{1\le i,j\le d})$ some polynomial functionals of $(f,\nabla f,\nabla^2 f)$ which satisfy

\begin{equation}\label{Nondege1}
\hat{\mathbb{P}}\text{-a.s.}, \int_{\mathbb{T}^d} \left(\Psi_0(x) \hat{f}(x)+\sum_{i=1}^d \Psi_i (x) \partial_i\hat{f}(x)+\sum_{i,j=1}^d \Psi_{i,j}(x)\partial_{i,j}\hat{f}(x)\right) dx=0.
\end{equation}

\subsubsection{Step1: Interpreting the conditions via the spectral measure}
We denote by $\mu$ the spectral measure associated with the Gaussian process $f$ and which is a probability measure on $\R^d$ supported on $\mathbb{Z}^d$ because our process is assumed to be periodic. We recall that the correlation function $\rho$ is defined via its spectral measure by 
\begin{equation}\label{spec-corr}
\E\left(f(x) f(y)\right)=\rho(y-x)=\sum_{\Vec{k}\in\mathbb{Z}^d} e^{\textbf{i} \Vec{k} \cdot (y-x)} \mu_{\Vec{k}}.
\end{equation}

By computing the variance of the equation (\ref{Nondege1}) with respect to $\hat{\mathbb{P}}$ and substituting the correlations functions by their spectral representation (\ref{spec-corr}), one can rewrite the equation (\ref{Nondege1}) in the following way

\begin{equation}\label{Nondege2}
\hat{\mathbb{P}}\text{-a.s.},\,\sum_{\Vec{k}\in\mathbb{Z}^d} \mu_{\Vec{k}} \left(\int_{\mathbb{T}^d}\left(\Psi_0(x)+\sum_{i=1}^d \textbf{i}\Psi_i (x) k_i-\sum_{i,j=1}^d k_i k_j \Psi_{i,j}(x)\right)e^{\textbf{i} \Vec{k}\cdot x} dx\right)^2=0.
\end{equation}

In order to understand the information contained in the equation (\ref{Nondege2}), we make the stronger assumption that the support of $\mu$ is the whole lattice $\mathbb{Z}^d$. In particular, we get that for every $\hat{\mathbb{P}}\text{-a.s.},$ and every $\Vec{k}\in\mathbb{Z}^d$:
\begin{equation}\label{Nondege3}
\int_{\mathbb{T}^d}\left(\Psi_0(x)+\sum_{i=1}^d \textbf{i}\Psi_i (x) k_i-\sum_{i,j=1}^d k_i k_j \Psi_{i,j}(x)\right)e^{\textbf{i} \Vec{k}\cdot x} dx=0.
\end{equation}
Now, taking $\theta \in \mathcal{C}^\infty_c(\mathbb{T}^d)$, one deduces from (\ref{Nondege3}) that
\begin{equation}
\int_{\mathbb{T}^d} \left(\Psi_0(x) \theta(x)+\sum_{i=1}\Psi_i(x) \partial_i \theta(x)+\sum_{i,j} \Psi_{i,j}(x) \partial_{i,j} \theta(x)\right) dx=0.
\end{equation}
\subsubsection{Step2: Reasoning on the set of positive values of $f$.} We take a test function in $\mathcal{C}_c^\infty(\mathbb{T}^d)$ but which is supported in the open set $\{f>0\}$. Recalling the specific form of $\Psi_0, (\Psi_i)_{i\ge 1}, (\Psi_{i,j})_{i,j\ge 1}$ which is given by the equation (\ref{form-auxiliary-functions}), one obtains for all $i\in\{1,\cdots,d\}$ and all $(k,l)\in\{1,\cdots,d\}^2$ that

$$\Psi_0\theta=\frac{P_0}{\eta_f^7}, \,\,\,\Psi_i \partial_i\theta=\frac{P_i}{\eta_f^7} \partial_i \theta,\,\,\,\Psi_{k,l} \partial_{k,l}\theta=\frac{P_{k,l}}{\eta_f^7} \partial_{k,l} \theta.$$

As a result, from the equation (\ref{Nondege3}), one derives that
\begin{equation}\label{Nondege4}
\int_{\mathbb{T}^d} \left(\frac{P_0(x)}{\eta_f^7(x)} \theta(x)+\sum_{i=1}\frac{P_i(x)}{\eta_f^7(x)} \partial_i \theta(x)+\sum_{i,j} \frac{P_{k,l}(x)}{\eta_f^7(x)} \partial_{i,j} \theta(x)\right) dx=0,
\end{equation}
as soon as $\theta$ is supported on the open set $\{f>0\}$. In order to fix the ideas, we make the stronger assumption that our process is of class $\mathcal{C}^\infty$ which implies in turn that the following functions $(\frac{P_0(x)}{\eta_f^7(x)}, (\frac{P_i(x)}{\eta_f^7(x)})_{i\ge 1}, (\frac{P_{k,l}(x)}{\eta_f^7(x)})_{i,j\ge 1})$ are of class $\mathcal{C}^\infty$. Performing two integrations by parts, equation (\ref{Nondege4}) leads to
\begin{equation}\label{Nondege5}
\int_{\mathbb{T}^d} \left(\frac{P_0(x)}{\eta_f^7(x)}+\sum_{i=1}\partial_i \left[\frac{P_i(x)}{\eta_f^7(x)}\right]+\sum_{i,j=1}^d \partial_{i,j} \left[\frac{P_{k,l}(x)}{\eta_f^7(x)}\right]\right) \theta(x)dx=0.
\end{equation}
Due to the polynomial nature of $(P_0,(P_i),(P_{i,j})$ the expression between parenthesis in the previous equation can be written in the form $$
\frac{R_1\Big(f,(\partial_i f)_{i},(\partial_{i,j} f)_{i,j},(\partial_{i,j,k} f)_{i,j,k},(\partial_{i,j,k,l} f)_{i,j,k,l}\Big)}{\eta^{11}},
$$
where $R_1$ is some explicit multivariate polynomial function. For the sake of simplicity we shall simply write:
$$\mathcal{R}_1:=R_1\Big(f,(\partial_i f)_{i},(\partial_{i,j} f)_{i,j},(\partial_{i,j,k} f)_{i,j,k},(\partial_{i,j,k,l} f)_{i,j,k,l}\Big).$$
The equation (\ref{Nondege5}) being valid for any test function supported on $\{f>0\}$ one deduces that $\mathcal{R}_1$ is zero on $\{f>0\}$. The same reasoning holds for the set $\{f<0\}$ and leads also to $\mathcal{R}_1=0$. Moreover, the set $\{f=0\}$ is of empty interior and each root of $f$ is in the adherence of $\{f>0\}\cup\{f<0\}$. Finally, the continuity of $\mathcal{R}_1$ entails that $\mathcal{R}_1=0$ on $\mathbb{T}^d$. Now coming back to the equation (\ref{Nondege5}) and performing again two integrations by parts one gets that for every $\theta \in \mathcal{C}^\infty(\mathbb{T}^d,\R)$, without any support restriction,
\begin{equation}\label{Nondege6}
\int_{\mathbb{T}^d} \left(\frac{P_0(x)}{\eta_f^7(x)} \theta(x)+\sum_{i=1}\frac{P_i(x)}{\eta_f^7(x)} \partial_i \theta(x)+\sum_{i,j} \frac{P_{k,l}(x)}{\eta_f^7(x)} \partial_{i,j} \theta(x)\right) dx=0,
\end{equation}
\subsubsection{Step3: A partial differential equation.}

Introduce now the partial differential operator:

$$
\mathcal{L}:~~\left\{
\begin{array}{rl}
\mathcal{C}^\infty(\mathbb{T}^d,\R)&\longrightarrow\mathcal{C}^\infty(\mathbb{T}^d,\R)\\
\theta&\longrightarrow \displaystyle{ \Psi_0(x) \theta(x)+\sum_{i=1}^d\Psi_i(x) \partial_i \theta(x)+\sum_{i,j=1}^d \Psi_{i,j}(x) \partial_{i,j} \theta(x)}.
\end{array}
\right.
$$
Denoting by $\mathcal{L}^*$ the adjoint of the operator $\mathcal{L}$, the combination of equations (\ref{Nondege3}) and (\ref{Nondege6}) leads to

$$\{1,\text{sign}(f)\}\subset \text{Ker} \left(\mathcal{L}^*\right).$$

As a result, one has the next corollary.

\begin{coro}\label{cor.nondeg}
If $\text{dim}\left(\text{Ker}\left(\mathcal{L}^*\right)\right)=1$ and the Malliavin derivative of the nodal volume vanishes, namely \[
\hat{\E}\left(\left[\s\mathcal H^{d-1}\left(\{f=0\}\right)\right]^2\right)=0,
\]
then $f$ has constant sign. 
\end{coro}

Due to the involved form of the coefficients of $\mathcal{L}$, it seems unclear for the authors whether there is an ellipticity property on the operator $\mathcal{L}$. As a result, we were not able to prove that $\text{dim}\left(\text{Ker}\left(\mathcal{L}^*\right)\right)=1$. Such a conclusion would have implied that $f$ has a constant sign on $\mathbb{T}^d$ if and only if the Malliavin derivative of the nodal volume is zero which is a very natural statement.

\subsection{Extension to the non-periodic setting}

In this section we briefly explain how to extend the content of Theorem \ref{Main theorem-periodic} to a non periodic setting and to obtain the announced Theorem \ref{Mainpasperio}. We give us a $\mathcal{C}^2$ and non periodic stationary Gaussian field on some domain $P:=[a_1,b_1]\times\cdots[a_d,b_d]\subset\R^d$. One is only left to show that the nodal volume $\mathcal{H}^{d-1}\left(f=0\right)$ is also in the domain of the Malliavin operator $\s$. The main difference is that the lack of periodicity creates many boundary terms when performing integrations by parts and one has to check that theses boundary terms belong also to the domain of the Malliavin derivative.

\subsubsection{A closed formula for the nodal volume involving boundary terms.} Our starting point is the proof of Proposition \ref{formule-IPP1}, from which we keep the notations. Let us fix some $i\in\{1,\cdots,d\}$ and denote by $P^{(i)}$ the Cartesian product of intervals $[a_j,b_j]$ where the coordinate $i$ has been removed and  by $dx^{(i)}$ the associated Lebesgue measure.

\begin{eqnarray*}
V_i^{\epsilon}&=&-\frac{1}{2} \int_{P} \left( \phi_{\epsilon} \circ f(x) \right) \left(  \frac{f^2\partial_{ii}^2 f - f(\partial_i f)^2+\partial_{ii}^2 f ||\nabla f||^2-\partial_i f \sum_j \partial_j f \partial_{ij}^2f}{\eta_f(x)^3}\right)  dx\\
&+& \underbrace{\frac{1}{2}\left[\int_{P^{(i)}} \phi_{\epsilon} \circ f(x)\left(  \frac{\partial_i f(x)}{\sqrt{|f(x)|^2 + ||\nabla f||^2}}\right)dx^{(i)]}\right]_{a_i}^{b_i}.}_{W^\epsilon_i}
\end{eqnarray*}

It still possible to let $\epsilon$ tend to zero and use the dominated convergence theorem. One recovers, after summation on the indexes $i\in\{1,\cdots,d\}$ that

\begin{eqnarray*}
V&=&-\frac{1}{2}\int_{P} |f(x)|\frac{\left(f(x)\Delta f(x)-\|\nabla f(x)\|^2\right)}{\eta_f(x)^3}dx\\
&-&\frac{1}{2} \int_{P} \text{sign}(f(x))\left(\Delta f(x) ||\nabla f(x)||^2 - \nabla f(x)^* \text{Hess}_x(f) \nabla f(x)\right)\frac{dx}{\eta_f(x)^3}\\
&+& \frac{1}{2}\sum_{i=1}^d \underbrace{\left[\int_{P^{(i)}}\text{sign}(f(x))\frac{\partial_i f(x)}{\eta_{f}^2(x)}dx^{(i)}\right]_{a_i}^{b_i}}_{W_i}.
\end{eqnarray*}

As in the proof of formula of Proposition \ref{prop.IPP2}, one has to remove the singular terms $\text{sign}(f)$. To do so, a supplementary integration by parts is required. The first singular term

$$\frac{1}{2} \int_{P} \text{sign}(f(x))\left(\Delta f(x) ||\nabla f(x)||^2 - \nabla f(x)^* \text{Hess}_x(f) \nabla f(x)\right)\frac{dx}{\eta_f(x)^3},$$

can be handled as in the proof of Proposition \ref{prop.IPP2}. Nevertheless, the boundary terms do not vanish anymore and one get the supplementary terms

\begin{eqnarray*}
\left[\int_{P^{(i)}} |f(x)| \frac{\Delta f(x) \partial_i f(x)-\sum_{j=1}^d \partial_j f(x) \partial_{i,j}^2 f(x)}{\eta_f (x)^3}dx^{(i)}\right]_{a_i}^{b_i}.
\end{eqnarray*}

In order to handle the other singular terms $W_i$ we need some preliminary observations. Recall that our Gaussian field is such that $(f(x),\nabla f(x))$ has a density. Fix some $i\in\{1,\cdots,d\}$, $a \in [a_i,b_i]$ and consider
\[
g_i(x):=g_i(x_1,x_2,\cdots,x_{i-1},x_{i+1},\cdots,x_d)=f(x_1,\cdots,x_{i-1},\textbf{a},x_{i+1},\cdots,x_d),
\]
the Gaussian field on $\R^{d-1}$ obtained by setting the $i$-th coordinate equal to $\textbf{a}$. Then, we infer that $(g_i(x),\nabla g_i(x))$ has a density on $\R^{d-1}$. Indeed, $g_i$ is self-evidently a stationary Gaussian process and its covariance matrix $\Sigma_{g_i}$ is obtained by removing the row and line $i$ from the covariance matrix $\Sigma_f$ of $(f,\nabla f)$. The Sylvester criterion ensures us that $\Sigma_g$ is invertible. As a consequence, $\eta_{g_i}>0$ for every $i\in\{1,\cdots,d\}$. Denote by $\delta_a$ the Dirac measure on $a$, for writing convenience we shall write the boundary terms occurring in the integrations by part as integrals with respect to tensor products of Lebesgue and Dirac measures. We have,
\begin{eqnarray*}
&&\sum_{i=1}^d \int_{\R}\left(\int_{P^{(i)}}\text{sign}(f(x))\frac{\partial_i f(x)}{\eta_f(x)}dx^{(i)}\right)\delta_a(x_i)\\
&&=\sum_{i=1}^d\int_{\R}\left(\int_{P^{(i)}}\text{sign}(f(x))\frac{\partial_i f(x)\eta_{g_i}^2(x)}{ \eta_{g_i}^2(x) \eta_f(x)}dx^{(i)}\right)\delta_a(x_i)\\
\end{eqnarray*}

Since we integrate the $i$-th coordinate with respect to $\delta_a$, we deduce that 

\begin{eqnarray*}
&&\int_{\R}\int_{P^{(i)}}\text{sign}(f(x))\frac{\partial_i f(x)\eta_{g_i}^2(x)}{\eta_{g_i}^2(x)\eta_{f}^2(x)}dx^{(i)}\delta_a(x_i)\\
&&=\int_{\R}\int_{P^{(i)}}|g_i(x)| \frac{\partial_i f(x) g_i(x)}{\eta_{g_i}^2(x)\eta_{f}^2(x)}dx^{(i)}\delta_a(x_i)\\
&&+\sum_{j=1,j \neq i}^d \underbrace{\int_{\R}\int_{P^{(i)}} \frac{\partial_j |g_i(x)| \partial_j g_i(x)\partial_i f(x) }{\eta_{g_i}^2(x)\eta_{f}^2(x)}dx^{(i)}\delta_a(x_i)}_{Z_{i,j}}.\\
\end{eqnarray*}

Integrating by parts leads to

\begin{eqnarray*}
&&Z_{i,j}=\int_{\R}\int_{P^{(i)}} \frac{\partial_j |g_i(x)| \partial_j g_i(x)\partial_i f(x) }{\eta_{g_i}^2(x)\eta_{f}^2(x)}dx^{(i)}\delta_a(x_i)\\
&&=\int_{\mathbb{R}^2}\int_{P^{(i,j)}}\frac{|g_i(x)| \partial_j g_i(x)\partial_i f(x) }{\eta_{g_i}^2(x)\eta_{f}^2(x)}dx^{(i,j)}\delta_a(x_i)(\delta_{b_j}(x_j)-\delta_{a_j}(x_j))\\
&&-\int_{\R}\int_{P^{(i)}} |g_i(x)|\partial_j\left[\frac{ \partial_j g_i(x)\partial_i f(x) }{\eta_{g_i}^2(x)\eta_{f}^2(x)}\right]dx^{(i)}\delta_a(x_i)
\end{eqnarray*}

\paragraph{Checking the integrability conditions for the sharp operator.} Clearly the terms in common with the formula corresponding to the periodic case produce the same integrability restrictions. One is only left to study each border term occurring the previous computations. 
Firstly, the terms of the form
\begin{eqnarray*}
\left[\int_{P^{(i)}} |f(x)| \frac{\Delta f(x) \partial_i f(x)-\sum_{j=1}^d \partial_j f(x) \partial_{i,j}^2 f(x)}{\eta_f (x)^3}dx^{(i)}\right]_{a_i}^{b_i}
\end{eqnarray*}
requires that $\eta_f^{-1}\in\mathbb{D}^{1,p}$ which holds true whenever $p<\frac{d+1}{3}$, see the proof of Lemma \ref{Integrand is in the domain}. One has to focus on the quantities $Z_{i,j}$. First the terms of the form
$$|g_i(x)| \frac{\partial_i f(x) g_i(x)}{\eta_{g_i}^2(x)\eta_{f}^2(x)}$$
are already bounded and taking the sharp operator requires a $p$-th moment for the quantity $\eta_{g_i}^{-1}$ which holds whenever $p<d$ and a fortiori $p<\frac{d+1}{3}$. As for the last integrand
$$|g_i(x)|\partial_j\left[\frac{ \partial_j g_i(x)\partial_i f(x) }{\eta_{g_i}^2(x)\eta_{f}^2(x)}\right],$$
it requires that $\eta_{g_i}^{-1}\in\mathbb{D}^{1,p}$ which in turn requires that $\eta_{g_i}^{-2}$ has a moment of order $p$. Following the same lines as in the proof of Lemma \ref{CR-technical} and noticing that $(g_i,\nabla g_i)$ is a $d$-dimensional Gaussian vector, one can take the polar coordinates to get the condition $d-1-2p>-1$ and so $p<\frac{d}{2}$. Finally, we notice that $\frac{d+1}{3}\le \frac{d}{2}$ as soon as $d\ge 3$ which explains why in both periodic and non-periodic frameworks, the integrability conditions are the same.
\bibliographystyle{alpha}
\newcommand{\etalchar}[1]{$^{#1}$}

\end{document}